\documentclass[12pt,reqno]{amsart}
  \usepackage{latexsym} 
  \usepackage[all]{xy}
  \usepackage{amsfonts} 
  \usepackage{amsthm} 
  \usepackage{amsmath} 
  \usepackage{amssymb}
  \usepackage{pifont}  
  \usepackage{enumerate}
  \usepackage{dcolumn}
  \usepackage{comment}
  \usepackage{hyperref}  
\newcolumntype{2}{D{.}{}{2.0}}
  \xyoption{2cell}

 \usepackage{tikz}
\usetikzlibrary{arrows}

  \def\su#1{{\sp{[#1]}}}

   \def\<{{\langle}} 
  \def\>{{\rangle}} 
  \def\ra{{\triangleleft}} 
  \def\la{{\triangleright}} 
  \def\eps{\varepsilon}

  \def\note#1{{}}

  \def\note#1{}

  \def\hom#1#2#3{{{\rm Hom}\sb{#1}(#2,#3)}} 
  \def\lend#1#2{{{\rm End}\sb{#1}(#2)}}

  \def\End#1#2{{{\rm End}\sb{#1}(#2)}}

  \def\ring{{\rm Ring}}

  \def\beq{\begin{equation}} 
  \def\eeq{\end{equation}}

  \def\id{\mathrm{id}} 
  \def\im{{\rm Im}}

 \def\pr{\mathrm{pr}}
 \def\tPsi{\widetilde{\Psi}}

  \newcounter{zlist} 
  \newenvironment{zlist}{\begin{list}{(\arabic{zlist})}{ 
  \usecounter{zlist}\leftmargin2.5em\labelwidth2em\labelsep0.5em 
  \topsep0.6ex
  \parsep0.3ex plus0.2ex minus0.1ex}}{\end{list}}

  \newcounter{blist} 
  \newenvironment{blist}{\begin{list}{(\alph{blist})}{ 
  \usecounter{blist}\leftmargin2.5em\labelwidth2em\labelsep0.5em 
  \topsep0.6ex 
  \parsep0.3ex plus0.2ex minus0.1ex}}{\end{list}} 

  \newcounter{rlist} 
  \newenvironment{rlist}{\begin{list}{(\roman{rlist})}{ 
  \usecounter{rlist}\leftmargin2.5em\labelwidth2em\labelsep0.5em 
  \topsep0.6ex 
  \parsep0.3ex plus0.2ex minus0.1ex}}{\end{list}}

\def\stac#1{\raise-.2cm\hbox{$\stackrel{\displaystyle\otimes}{\scriptscriptstyle{#1}}$}}

\def\cten#1{\raise-.2cm\hbox{$\stackrel{\displaystyle\reallywidehat{\otimes}}
{\scriptscriptstyle{#1}}$}}

  \def\Label#1{\label{#1}\ifmmode\llap{[#1] }\else 
  \marginpar{\smash{\hbox{\tiny [#1]}}}\fi} 
  \def\Label{\label}

  \newtheorem{proposition}{Proposition}[section]
  \newtheorem{lemma}[proposition]{Lemma} 
  \newtheorem{corollary}[proposition]{Corollary} 
  \newtheorem{theorem}[proposition]{Theorem} 
  
\theoremstyle{definition} 
  \newtheorem{definition}[proposition]{Definition}
    
  \newtheorem{example}[proposition]{Example}

  \theoremstyle{remark} 
  \newtheorem{remark}[proposition]{Remark}

  \newcounter{c} 
   
  \newcommand{\etyk}[1]{\vspace{-7.4mm}$$\begin{equation}\Label{#1} 
  \addtocounter{c}{1}} 
  \renewcommand{\]}{\ifnum \value{c}=1 $$\else \end{equation}\fi} 
   \numberwithin{equation}{section}

\def\FF{{\mathbb F}}

\def\NN{{\mathbb N}}

\def\ZZ{{\mathbb Z}}

\newcommand{\gG}{\mathrm{G}}
\newcommand{\hH}{\mathrm{H}}
\newcommand{\iI}{\mathrm{I}}

\newcommand{\rR}{\mathrm{R}}

\newcommand{\tT}{\mathrm{T}}

\newcommand{\Cc}{\mathcal{C}}

\newcommand{\treq}{\overset{tr}{\cong}}

\def\A{{\bf A}}
\def\B{{\bf B}}

\def\*C{{}^*\hspace*{-1pt}{\Cc}}

\def\text#1{{\rm {\rm #1}}}

 \def\1{\mathbf{1}}

  \def\la#1{\stackrel{#1}{\triangleright}}
 
   \def\ra#1{\stackrel{#1}\triangleleft}
\def\id{\mathrm{id}}

\def\di{\diamond}

\def\lto{\longmapsto}
\def\lra{\longrightarrow}

\def\pr#1{\bullet_{#1}}

    \def\eps{\varepsilon}

\def\1\mathbf{1}

\def\|#1{\overline{#1}}

\def\linv#1#2{\mathrm{Linv}(#1;#2)}
\def\rinv#1#2{\mathrm{Rinv}(#1;#2)}
\def\sige{\varepsilon}
\def\sigo{\sigma}
\def\sigol{\stackrel{\leftarrow}\sigma}
\def\sigor{\stackrel{\rightarrow}\sigma}
\def\truss {\mathbf{Trs}}
\def\ring {\mathbf{Ring}}

\usepackage{scalerel,stackengine}
\stackMath
\newcommand\reallywidehat[1]{%
\savestack{\tmpbox}{\stretchto{%
  \scaleto{%
    \scalerel*[\widthof{\ensuremath{#1}}]{\kern.1pt\mathchar"0362\kern.1pt}%
    {\rule{0ex}{\textheight}}
  }{\textheight}%
}{2.4ex}}%
\stackon[-6.9pt]{#1}{\tmpbox}%
}
\begin{document}

\title{Ideal ring extensions and trusses}

\author{Ryszard R.\ Andruszkiewicz}

\address{
Faculty of Mathematics, University of Bia{\l}ystok, K.\ Cio{\l}kowskiego  1M,
15-245 Bia\-{\l}ys\-tok, Poland}

\email{randrusz@math.uwb.edu.pl}

\author{Tomasz Brzezi\'nski}

\address{
Department of Mathematics, Swansea University, 
Swansea University Bay Campus,
Fabian Way,
Swansea,
  Swansea SA1 8EN, U.K.\ \newline \indent
Faculty of Mathematics, University of Bia{\l}ystok, K.\ Cio{\l}kowskiego  1M,
15-245 Bia\-{\l}ys\-tok, Poland}

\email{T.Brzezinski@swansea.ac.uk}

\author{Bernard  Rybo{\l}owicz}

\address{
Department of Mathematics, Swansea University, 
Swansea University Bay Campus,
Fabian Way,
Swansea,
  Swansea SA1 8EN, U.K.}

\urladdr{https://sites.google.com/view/bernardrybolowicz/}
\email{Bernard.Rybolowicz@swansea.ac.uk}



\begin{abstract}
It is shown that there is a close relationship between ideal extensions of rings and trusses, that is, sets with a semigroup operation distributing over a ternary abelian heap operation.  Specifically, a truss can be associated to every element of an extension ring that projects down to an idempotent in the extending ring; every weak equivalence of extensions yields an isomorphism of corresponding trusses. Furthermore, equivalence classes of  ideal extensions of rings  by integers 
are in one-to-one correspondence with  associated trusses up to  isomorphism given by a translation. Conversely, to any truss $T$ and an element of this truss one can associate a ring and its extension by integers 
in which $T$ is embedded as a truss. Consequently any truss can be understood as arising from an ideal extension by integers. The key role is played by  interpretation of ideal extensions by integers as extensions defined by double homothetisms of Redei [L.\ Redei, Die Verallgemeinerung der Schreierschen Erweiterungstheorie, {\em Acta Sci.\ Math.\ Szeged},  {\bf 14} (1952),  252--273] or by self-permutable bimultiplications of Mac Lane [S.\ Mac Lane, Extensions and obstructions for rings, {\em Illinois J.\ Math.} {\bf 2} (1958), 316--345], that is,  as {\em integral homothetic extensions}.  It is shown that integral homothetic extensions of trusses are universal as extensions of trusses to rings but still enjoy a particular smallness property: they do not contain any subrings to which the truss inclusion map corestricts. Minimal extensions of trusses into rings are defined. The correspondence between homothetic ring extensions and trusses is used to classify fully up to isomorphism  trusses arising from rings with zero multiplication and rings with trivial annihilators.
\end{abstract}    
\date\today
\maketitle
\tableofcontents

\part{Prelude}\label{part.pre}
\section{Introduction}
A truss \cite{Brz:tru}, \cite{Brz:par} is an algebraic system consisting of a set with a ternary operation making it into an abelian heap \cite{Pru:the}, \cite{Bae:ein} and an associative binary operation that distributes over the ternary one. From the universal algebra point of view its composition involves less operations, so it is simpler, than that of a ring (which consists of two binary operations, one unary operation and one nullary operation) or a  (two-sided) brace \cite{Rum:bra}, \cite{CedJes:bra}, \cite{GuaVen:ske}, which as a set with two entangled group structures involves  six operations and whose connection with the set-theoretic Yang-Baxter equation has led to a remarkable surge in interest recently.  Yet a truss equipped with  a specific nullary operation or with an element with special properties can be made into a ring; if the binary operation is a group operation, then there is a (two-sided) brace associated to a truss. Conversely, every ring can be made into a truss in a natural way, by associating the (unique) heap operation $[a,b,c] =a-b+c$ to the abelian group operation, and so  can every brace. Thus trusses are both simpler in architecture and more general than rings. Alas their definition involves a ternary operation that is less intuitive and familiar than binary operations, and so one is faced with a familiar dilemma of generality versus comprehension.

In mathematics as in any other human endeavour  there is a natural tendency to familiarise what is new or unknown by contrasting or comparing it with what is well-known. Thus  the desire to see how trusses are related to rings or how trusses can be described in ring-theoretic terms  is most understandable. The key observation of Rump \cite{Rum:bra} that a two-sided brace can be made into a radical ring by modifying one of its operations (and vice versa, a radical ring gives rise to a two-sided brace) has been extended to trusses in \cite{Brz:tru}, \cite{Brz:par} with a caveat that a central element must exist; the resulting ring is not necessarily a radical ring. In this paper we show that a ring can be associated to any truss and any element, but we go further than that. We show that a natural framework for ring-theoretic studies of trusses is provided by ideal ring extensions (see, for example, \cite{Pet:ide}), in particular those that arise from Redei's homothetisms \cite{Red:ver} or Mac Lane's  self-permutable bimultiplications \cite{Mac:ext}. In short, we show that a truss can be associated to every element of an extension ring that projects down to an idempotent in the extending ring, while every weak equivalence of extensions yields an isomorphism of corresponding trusses. Furthermore,  equivalence classes of extensions of rings by integers are in one-to-one correspondence with the class of associated trusses up to isomorphism given by translation. Conversely, to any truss $T$ and an element of this truss one can associate a ring together with its extension by integers  in which $T$ is embedded as a truss (and also as a paragon or an equivalence class of a congruence of rings). Since every ideal extension of a ring by integers is equivalent to an extension by $\ZZ$ through a double homothetism, any truss can be understood as arising from such a homothetic extension, i.e.\ every  truss is a {\em homothetic truss}. Note in passing that {\em all} (not only those by integers) ideal ring extensions arise from families of permutable bimultiplications or amenable homothetisms by the Everett theorem \cite{Eve:ext} (see \cite{Pet:ide} for an elegant presentation and simplification of the proof).

The paper is divided into four parts and a short coda. The first part (which includes this introduction) contains preliminary definitions, in particular basic notions from the theory of heaps and trusses. The second part gathers main results of the paper. We begin by recalling the definitions of ring extensions and their equivalences. An extension $(\varphi_R, S,\varphi_Z)$ consists of a ring monomorphism $\varphi_R: R\lra S$ and a ring epimorphism $\varphi_Z: S\lra Z$ such that $\ker \varphi_Z = \im \varphi_R$ (unless stated otherwise, by a ring we mean an associative ring not necessarily with identity).  Following the terminology of Mac Lane \cite{Mac:ext} we refer to such an extension as to an extension of $R$ to $S$ by $Z$ (note that an opposite terminology is often used in homological algebra). Two extensions $(\varphi_R, S,\varphi_Z)$ and $(\varphi'_R, S',\varphi'_Z)$ are equivalent if there is an isomorphism between $S$ and $S'$ that commutes with the identity automorphisms on $R$ and $Z$. Furthermore, we say that $(\varphi_R, S,\varphi_Z)$ and $(\varphi'_R, S',\varphi'_Z)$ are weakly equivalent, provided that they are equivalent up to an automorphism of $R$. We observe  in Proposition~\ref{prop.truss.ext} that if there is $q\in S$ such that $q^2-q \in \varphi_R(R)$, then $q+\varphi_R(R)$ is a sub-truss of the truss associated to $S$. Furthermore, any  weak equivalence map for extensions restricts to an isomorphism of corresponding trusses. We then proceed to focus on a class of extensions arising from double homothetisms or self-permutable bimultiplications, i.e.\ pairs of additive endomorphisms  $\sigma = (\sigor, \sigol)$ of a ring $R$ satisfying a number of associativity-like conditions; see Definition~\ref{def.homothetism}. In Theorem~\ref{thm.h.ext} we associate an extension of a ring $R$ by integers (or by integers modulo the exponent of the abelian group of $R$ if finite) to a pair $(\sigma, s)$ consisting of a double homothetism $\sigma$ on $R$ and an element $s\in R$ such that $\sigma s = s\sigma$ and $\sigma^2 - \sigma$ is the bimultiplication by $s$. We term the resulting extension of $R$, denoted by $R(\sigma,s)$, a homothetic extension of $R$ (infinite in the integral case and cyclic or finite in the modular case). Next we note that in fact every ring extension by $\ZZ$ is equivalent to an infinite homothetic extension and that every $R(\sigma, s)$ induces a truss $T(\sigma, s)$ on the heap corresponding to the abelian group of $R$. We term $T(\sigma, s)$ a {\em homothetic truss}. In a couple of lemmas leading to Theorem~\ref{thm.h.equiv} we describe isomorphisms of homothetic trusses and connect them with (weak) equivalences of homothetic extensions: Two extensions are weakly equivalent if and only if corresponding trusses are isomorphic, while equivalences of extensions correspond to isomorphisms of trusses given by translations by an element.

In Section~\ref{sec.truss-ext} we take the opposite route: from trusses to ring extension. The first main result of this section is Theorem~\ref{thm.truss.ring}, which allows one to associate a ring to any truss $T$ and an element $e\in T$. We denote this truss by $\rR(T;e)$. This ring has addition obtained by retracting of the ternary operation in $T$ at $e$ (so it has an abelian group structure derived from the heap structure of $T$), but with a modified multiplication. There are no assumptions on $e$, and thus Theorem~\ref{thm.truss.ring} provides one with a generalisation of the construction presented in \cite{Brz:tru}, which associates a ring to a central element of a truss or the construction of Rump \cite{Rum:bra} connecting two-sided braces with radical rings. The way to recover the original multiplication of $T$ is described in Theorem~\ref{thm.hom}: $e$ determines a homothetic datum $(\eps, e^2)$ on $\rR(T;e)$, and $T$ is the truss embedded in the corresponding homothetic ring extension of $\rR(T;e)$, denoted by $T(e)$ in the infinite case or $T^\mathrm{c}(e)$ in the cyclic case, that is, $T=\tT(\eps,e^2)$. The results of Part~\ref{part.ext} thus can be summarised as one of two main messages of this paper: {\bf\em every truss is a homothetic truss;
every ring is a ring associated to a truss with a fixed element.}

Part~\ref{part.int} on one hand gives a categorical interpretation of the infinite homothetic ring extensions $T(e)$ in which a truss $T$ embeds, while on the other attempts at characterisation of smallest rings into which $T$ embeds. In Theorem~\ref{thm.ext.0} we construct a ring isomorphism between $T(e)$ and the universal ring $T_0$ obtained from a truss $T$ by adjoining the zero element \cite[Lemma~3.13]{BrzRyb:mod}. As a consequence, the infinite homothetic extension $T(e)$ has the following universal property: any truss homomorphism from $T$ to a ring $R$ factorises through the inclusion $T\hookrightarrow T(e)$ and a unique ring homomorphism $T(e)\lra R$, i.e. there exists a universal arrow from $T$ to the functor $\tT: \ring\lra \truss$ (see \cite[Section~III.1]{Mac:lane}). Section~\ref{sec.min} outlines various possible notions of `smallness' of a ring that contains $T$. The most crude one is proposed in Definition~\ref{def.locsmal}: a locally small extension of $T$ is given by a ring which has no proper subrings containing $T$ (as a sub-truss).  The infinite homothetic extension of $T$ into ring $T(e)$ has this property. The intermediate notion given in Definition~\ref{def.small}  distinguishes those small extensions in which $T$ is isomorphic (as a heap) with an essential ideal. Finally, a minimal extension corresponds to ideals in $T(e)$ which intersect trivially with the ideal induced by the canonical inclusion of $T$ into $T(e)$. All these notions and differences between them are illustrated by examples.

Part~\ref{part.class} uses results of Part~\ref{part.ext}, in particular that every truss is a homothetic truss, to classify all trusses corresponding to rings with zero multiplication (Section~\ref{sec.zero}) and zero annihilators (Section~\ref{sec.ann}).  In the first case, in which all the structural information is necessarily contained in the abelian group of the ring, we show that there is a one-to-one correspondence between isomorphism classes of trusses corresponding to rings with zero multiplication on an abelian group $A$ and ordered direct sum decompositions of $A$ into four subgroups; Theorem~\ref{thm.zero}. In the latter case there is a  bijective correspondence between isomorphism classes of homothetic trusses on $R$ and equivalence classes of idempotents in the ring $\Xi(R)$ of outer bimultiplications on $R$ (with respect to the relation defined in Definition~\ref{def.out.equiv}); see Theorem~\ref{thm.ann}. In particular, and quite surprising, there are exactly two isomorphism classes of trusses on one-sided maximal ideals in simple rings with identity; see Theorem~\ref{thm.simple}. This seems to be a rather unexpected application of the theory of maximal essential extensions developed by Beidar \cite{Bei:ato}, \cite{Bei:ess}. All the results of this part are employed to give a full classification of non-isomorphic trusses with the heap structure corresponding to the abelian group $\ZZ_p\times \ZZ_p$.

\section{Preliminaries}\label{sec.prem}
We start by gathering in one place key information about heaps and trusses, and by establishing the notation. Further details can be found in e.g.\ \cite{Brz:par}.

An abelian heap is a set $H$ with a ternary operation $[-,-,-]$ such that, for all $h_i$, $i=1,\ldots, 5$.
\begin{subequations}\label{heap}
\begin{equation}\label{assoc}
[h_1,h_2, [h_3,h_4,h_5]] = [[h_1,h_2, h_3],h_4,h_5],
\end{equation}
\begin{equation}\label{Malcev}
[h_1,h_1,h_2] = h_2 \quad \&  \quad [h_1,h_2,h_2] = h_1,
\end{equation}
\begin{equation}\label{comm}
[h_1,h_2,h_3] = [h_3,h_2,h_1].
\end{equation}
\end{subequations}
Equation \eqref{assoc} expresses the associative law for heaps, equations \eqref{Malcev} are known as Mal'cev identities, and equation \eqref{comm} is the heap commutative law. In view of these axioms the distribution of brackets in multiple applications of the ternary operation does not play any role, and hence we write
$$
[h_1,h_2,\ldots, h_{2n+1}]
$$
for the element of $H$ obtained by any possible application of the ternary operation to the (always odd) $2n+1$-tuple $(h_1,h_2,\ldots, h_{2n+1}) \in H^{2n+1}$. Furthermore, equations \eqref{Malcev} and \eqref{comm} yield the following cancellation and rearrangement rules, for all $h_1,h_2,\ldots, h_{2n+1}\in H$,
\begin{subequations}\label{rear}
\begin{equation}\label{cancel}
[h_1,\ldots, h_{i-1},h_i,h_i,h_{i+1},\ldots h_{2n}] = [h_1,\ldots, h_{i-1},h_{i+1},\ldots h_{2n}],
\end{equation}
\begin{equation}\label{symm}
[h_1,h_2,\ldots, h_{2n+1}] = [h_{\varpi(1)},h_{\varsigma(2)},h_{\varpi(3)},\ldots, h_{\varsigma(2n)},h_{\varpi(2n+1)}],
\end{equation}
\end{subequations}
for any permutation $\varpi$ on the set $\{1,3,5,\ldots, 2n+1\}$ and any permutation $\varsigma$ on $\{2,4,\ldots, 2n\}$. 

For any $e\in H$, the set $H$ with the binary operation $+_e = [-,e,-]$ is an abelian group, known as a retract of $H$. The chosen element $e$ is  the zero for this group and the inverse $-_eh$ of $h$ is $[e,h,e]$. We denote this unique up to isomorphism group by $\gG(H;e)$. Conversely, for any (abelian) group $G$, the operation $[a,b,c] = a-b+c$ defines the  heap structure on $G$; we denote this heap by $\hH(G)$. A homomorphism of heaps is a function that preserves heap operations. In particular, any group homomorphism is a heap homomorphism for the corresponding heaps so the assignment $\hH: G\lto \hH(G)$ is a functor from the category of (abelian) groups to the category of (abelian) heaps.

A truss is a set $T$ together with a ternary operation $[-,-,-]$ and a binary operation $\cdot$ (denoted by a juxtaposition of elements and called  multiplication) such that $(T,[-,-,-])$ is an abelian heap, $(T,\cdot)$ is a semigroup and $\cdot$ distributes over $[-,-,-]$, that is, for all $a,b,c,d\in T$,
\begin{equation}\label{dist}
a[b,c,d] = [ab, ac,ad] \quad \& \quad [a,b,c]d = [ad,bd,cd].
\end{equation}
A morphism of trusses is a function that is a homomorphism of both heaps and semigroups.
Unless stated otherwise, by a ring we mean an associative ring not necessarily with identity. To any ring $(R,+,\cdot)$ one can associate a truss $\tT(R)$ with the heap structure $\hH(R,+)$ and the original multiplication.  The assignment $\tT: R\lto \tT(R)$ and identity on morphisms is a functor from the category of rings to the category of trusses. Conversely, if a truss $T$ has an {\em absorber}, that is, an element $e$ such that $ea=ae=e$, for all $a\in T$, then $\rR(T;e) := (\gG(R;e),\cdot)$ is a ring and $T= \tT(\rR(T;e))$. 

An equivalence class of a congruence in a truss $T$ is called a {\em paragon}. Equivalently, a paragon $P$ is a sub-heap of $(T,[-,-,-])$ such that, for all $p,q\in P$ and $a\in T$,
\begin{equation}\label{par.def}
[ap,aq,q]\in P \quad \& \quad [pa,qa,q]\in P.
\end{equation}
A sub-heap satisfying the first of equations of \eqref{par.def} is called a {\em left paragon} and one that satisfies the second of these equations is called a {\em right paragon}.

\part{Extensions}\label{part.ext}

\section{From ideal extensions of rings to trusses}\label{sec.ext-truss}
This section contains the first main results of the paper. We begin by recalling the notions of ideal extensions of rings and (weak) equivalences between such extensions. Next we show that one can associate a truss to every element of an extension ring that yields an idempotent in the ring by which the extension is achieved. Weak equivalences of extensions restrict to isomorphisms  of these trusses. In the converse direction we construct (infinite and finite) extensions of a given ring by double homothetisms, and show that such extensions are weakly equivalent if and only if the associated trusses are isomorphic and equivalent if and only if the associated trusses are isomorphic by a translation or {\em translationally isomorphic}. We also observe that any infinite homothetic extension is equivalent to an extension by the ring of integers.

The following definition  is taken from \cite{Mac:ext} (see also \cite{Pet:ide}).

\begin{definition}\label{def.ring.ext}
An exact sequence of ring homomorphisms
$$
\xymatrix{0 \ar[r] & R \ar[r]^{\varphi_R} & S\ar[r]^{\varphi_Z} & Z \ar[r] & 0,}
$$
is called an {\em ideal ring extension of $R$ to $S$ by $Z$} or simply a {\em ring extension}. We write $({\varphi_R}, S, {\varphi_Z})$ for this ring extension.

Two extensions $({\varphi_R}, S, {\varphi_Z})$ and $({\varphi'_R}, S', {\varphi'_Z})$ are said to be  {\em equivalent} if there exists a ring isomorphism $\Theta: S\lra S'$ rendering the following diagram commutative
\begin{equation}\label{ext.equiv}
\xymatrix{R\ar[rr]^{\varphi_R} \ar[d]_{\varphi'_R} && S \ar[d]^{\varphi_Z}\ar[dll]^\Theta\\
S'\ar[rr]_{\varphi'_Z}&& Z.}
\end{equation}
We write $({\varphi_R}, S, {\varphi_Z}) \stackrel{\Theta}{\equiv} ({\varphi'_R}, S', {\varphi'_Z})$.

Two extensions $({\varphi_R}, S, {\varphi_Z})$ and $({\varphi'_R}, S', {\varphi'_Z})$ are said to be  {\em weakly equivalent} if there exist a ring isomorphism $\Theta: S\lra S'$  and a ring automorphism $\Theta_R:R\lra R$ such that $({\varphi_R}, S, {\varphi_Z}) \stackrel{\Theta}{\equiv} (\varphi'_R\circ \Theta_R, S', {\varphi'_Z})$. In that case we write $({\varphi_R}, S, {\varphi_Z}) \stackrel{\Theta}{\cong} (\varphi'_R, S', {\varphi'_Z})$
\end{definition}

The reader should be made aware that particularly in the texts in homological algebra (see e.g.\ \cite[Section~9.3]{Wei:int}) the sequence in Definition~\ref{def.ring.ext} is referred to as an extension of $Z$ to $S$ by $R$. We have elected here to chose the conventions of Mac Lane \cite{Mac:ext}.

\begin{proposition}\label{prop.truss.ext}
Let $(\varphi_R,S,\varphi_Z)$ be an ideal ring extension and let $q\in S$. Then
\begin{zlist}
\item The set $q+\varphi_R(R)$ is a sub-heap of $\hH(S)$ and a paragon of $\tT(S)$.
\item The sub-heap $q+\varphi_R(R)$ is a sub-truss of $\tT(S)$ if and only if
\begin{equation}\label{q.cond}
q^2 - q \in \varphi_R(R)
\end{equation}
This truss is denoted by $T(\varphi_R,S,\varphi_Z;q)$.
\item If  $({\varphi_R}, S, {\varphi_Z}) \stackrel{\Theta}{\cong} ({\varphi'_R}, S, {\varphi'_Z})$, then the map $\Theta$ restricts to the isomorphism of trusses $T(\varphi_R,S,\varphi_Z;q)\cong T(\varphi'_R,S',\varphi'_Z;\Theta(q))$. 
\end{zlist}
\end{proposition}
\begin{proof}
Since $\varphi_R(R)=\ker\varphi_Z$ is an ideal in $S$, for all $q\in S$, $q+\varphi_R(R)$ is an equivalence class of a congruence relation in $S$, hence it is a paragon in $\tT(S)$ by \cite[Corollary~3.3]{BrzRyb:con}. Condition \eqref{q.cond} is equivalent to the statement that $\varphi_Z(q) = q+\varphi_R(R)$ is an idempotent in $Z$, hence $(q+\varphi_R(R))^2  = q+\varphi_R(R)$ is closed under the multiplication in $S$.

Since $\Theta$ is a ring isomorphism, its restriction to $T(\varphi_R,S,\varphi_Z;q)$ is a monomorphism of trusses. We need to show that $\Theta(T(\varphi_R,S,\varphi_Z;q)) = T(\varphi'_R,S',\varphi'_Z;\Theta(q))$. For all $r\in R$,
$$
\Theta(q+\varphi_R(r)) = \Theta(q) + \Theta(\varphi_R(r)) = \Theta(q) +\varphi'_R(\Theta_R(r)) \in T(\varphi'_R,S',\varphi'_Z;\Theta(q)).
$$
The commutativity of the diagram \eqref{ext.equiv} and the fact that $\Theta_R$ is an automorphism of rings  yield the surjectivity of $\Theta$. This completes the proof of the proposition.
\end{proof}

Since weakly equivalent extensions give rise to isomorphic trusses, Proposition~\ref{prop.truss.ext} provides one with a method of constructing (isomorphism classes of) trusses. In fact one can look for a statement in the opposite direction, that is, for types of extensions whose (weak) equivalence classes are determined by (isomorphism classes of) corresponding trusses. To this end we need to focus on extensions of a more specific kind.

\begin{definition}\label{def.homothetism}
Let $R$ be a ring and let $\sigo$  be a double operator on $R$, that is a pair of additive endomorphisms,
$$
\sigor:R\lra R,\quad a\lto \sigo a,\qquad \sigol:R\lra R,\quad a\lto a\sigo .
$$
\begin{zlist}
\item The double operator $\sigo$ is called a  {\em bimultiplication} \cite{Mac:ext} or a {\em bitranslation} \cite{Pet:ide} if, for all $a,b\in R$,
\begin{subequations}\label{bimult}
\begin{equation} \label{h.linear}
\sigo(ab) = (\sigo a)b \quad \& \quad (ab)\sigo = a( b\sigo),
\end{equation}
\begin{equation}\label{h.assoc}
a(\sigo b) = (a\sigo) b.
\end{equation}
\end{subequations}
The set of all bimultiplications is denoted by $\Omega(R)$.
\item A bimultiplication $\sigo$ is called a
{\em double homothetism} \cite{Red:ver} or is said to be  {\em self-permutable} \cite{Mac:ext} provided that, for all $a\in R$,
\begin{equation}\label{h.comm}
(\sigo a) \sigo =\sigo (a \sigo) .
\end{equation}
The set  of all double homothetisms on $R$ is denoted by $\Pi(R)$. 
\end{zlist}
\end{definition}

In short, conditions \eqref{h.linear} mean that $\sigor$ is a right and $\sigol$ is a left $R$-module homomorphism. A bimultiplication is called simply a   {\em multiplication} in  \cite{Hoc:coh}. In functional analysis, in particular in the context of $C^*$-algebras, bimultiplications are known as {\em multipliers} \cite{Hel:mul}, \cite{Bus:dou}. The relations  \eqref{h.comm} mean that  $\sigor$ commutes with $\sigol$ in the endomorphism ring $\lend{}{R,+}$. 
The set $\Omega(R)$ is a unital ring with the addition and multiplication, for all $\sigo, \sigo'\in \Omega(R)$, $a\in R$,
\begin{subequations}\label{ring.omega}
\begin{equation}\label{+omega}
(\sigo+\sigo')a = \sigo a + \sigo' a, \qquad a(\sigo+\sigo') = a\sigo  + a\sigo',
\end{equation}
\begin{equation}\label{x.omega}
(\sigo\sigo') a = \sigo(\sigo'a), \qquad a(\sigo\sigo')  = (a\sigo)\sigo'.
\end{equation}
\end{subequations}
In particular in the context of $C^*$-algebras, $\Omega(R)$ is known as a {\em multiplier algebra}. Note that the rules \eqref{x.omega} mean the composition of right-linear components of bimultiplictions and {\em opposite} composition of the left-linear ones. It is clear from \eqref{ring.omega} that $R$ is an $\Omega(R)$-bimodule.

In general, $\Pi(R)$ need not be a subring of $\Omega(R)$. 
The rules of Definition~\ref{def.homothetism} mean that we need not write any brackets in-between letters of the words composed of elements of $R$ and a homothetism on $R$.

For any $a\in R$, the left and right multiplications by $a$ form a double homothetism, which we denote by $\bar{a}$. That is, for all $b\in R$,
\begin{equation}\label{inner.hom}
\bar{a} b = ab, \qquad b\bar{a} = ba.
\end{equation}
Such a double homothetism is said to be {\em inner} and the abelian group of all inner homothetisms is denote by $\bar{R}$. The right $R$-module and left $R$-module components of $\bar{a}$ are denoted by $\stackrel{\rightarrow}a$ and $\stackrel{\leftarrow}a$ respectively. Note that if $\sigo$ is a double homothetism, then,   for all $a\in R$, $\sigo +\bar{a}$ is also a double homothetism by equations \eqref{bimult}. By the same token, for all $a\in R$ and $\sigma \in \Omega(R)$,
\begin{equation}\label{inn.ideal}
\bar{a}\sigo = \overline{a\sigma} \quad \& \quad \sigo\bar{a} = \overline{\sigo a}.
\end{equation}
This implies that $\bar{R}$ is an essential ideal in $\Omega(R)$. The quotient ring $\Omega(R)/\bar{R}$ is called the ring of {\em outer bimultiplications} and is denoted by $\Xi(R)$ (in context of $C^*$-algebras, $\Xi(R)$ might be referred to as a {\em corona algebra}). The canonical surjection $\Omega(R)\lra \Xi(R)$ is denoted by $\xi$.

The following lemma can be proven by direct calculations.
\begin{lemma}\label{lem.h.aut}
Let $\Phi$ be an automorphism of a ring $R$. For any bimultiplication $\sigo \in \Omega(R)$ define the double operator ${\Phi^*(\sigma)}$ on $R$ by
\begin{equation}\label{adj}
\stackrel{\lra}{\Phi^*(\sigma)}: a\lto \Phi(\sigo \Phi^{-1}(a)), \qquad \stackrel{\longleftarrow}{\Phi^*(\sigma)}: a \lto\Phi(\Phi^{-1}(a)\sigo ).
\end{equation}
The assignment $\sigma \lto \Phi^*(\sigma)$ defines a ring automorphism  $\Phi^*$ on $\Omega(R)$. Furthermore, $\Phi^*(\Pi(R)) = \Pi(R)$, $\Phi^*$ maps inner homothetisms to inner ones, and, for all $a\in R$,
\begin{equation}\label{adj.inner}
\overline{\Phi(a)} = \Phi^*(\bar{a}).
\end{equation}
\end{lemma}

Equation \eqref{adj.inner} implies that $\Phi^*:\Omega(R)\lra \Omega(R)$ descends to the ring automorphism $\Phi^\di : \Xi(R)\lra \Xi(R)$ by the diagram
\begin{equation}\label{diag.phi.di}
\xymatrix{ \Omega(R)\ar[rr]^-{\Phi^*} \ar[d]_\xi &&  \Omega(R)\ar[d]_\xi \\
\Xi(R)\ar[rr]^-{{\Phi^\di}} &&  \Xi(R),
}
\end{equation}
where $\xi$  is the canonical surjection. The following notion will prove particularly helpful for discussing trusses associated to rings with trivial annihilators.

\begin{definition}\label{def.out.equiv}
Two outer bimultiplications $\sigo, \sigo'\in \Xi(R)$ are said to be {\em equivalent} if there exists a ring automorphism $\Phi: R\lra R$ such that 
$
\sigma' ={\Phi^\di}(\sigma).
$
In that case we write $\sigo\sim\sigo'$.
\end{definition}

With all these notions and notation at hand we are now ready to consider extensions of our particular interest.

\begin{theorem}\label{thm.h.ext}
Let $R$ be a ring, and $\sigma \in \Pi(R)$ and $s \in R$ such that, 
\begin{subequations}\label{h.ext.cond}
\begin{equation}\label{h.norm}
\sigma s = s \sigma,
\end{equation}
\begin{equation}\label{h.sig}
\sigma^2   = \sigma + \bar{s}.
\end{equation}
\end{subequations}
Then
\begin{zlist}
\item 
\begin{blist}
\item The abelian group $R\times \ZZ$ together with the product, for all $a,b\in R$, $k,l\in \ZZ$,
\begin{equation}\label{prod.Z}
(a,k)(b,l) = \left(ab +l a\sigma + k\sigma b + kls, kl\right),
\end{equation}
is an associative ring. This ring is denoted by $R(\sigma,s)$.
\item The sequence
\begin{equation}\label{seq.inf}
\xymatrix{0 \ar[r] & R \ar[r]^-{\varphi_R} & R(\sigma,s) \ar[r]^-{\varphi_\ZZ} & \ZZ \ar[r] & 0,}
\end{equation}
where $\varphi_R: a\lto (a,0)$ and $\varphi_\ZZ: (a,k)\lto k$ is an exact sequence of rings.
\item Any ideal ring extension
\begin{equation}\label{seq.z}
\xymatrix{0 \ar[r] & R \ar[r]^-{\psi_R} & S \ar[r]^-{\psi_\ZZ} & \ZZ \ar[r] & 0.}
\end{equation}
is equivalent to an extension of type \eqref{seq.inf}.
\end{blist}
\item If the abelian group $(R,+)$ has a finite exponent $N$, then:
\begin{blist}
\item  The abelian subgroup 
$
I_N = \{0\}\times N\ZZ$ of $R\times \ZZ$ 
is an ideal in $R(\sigma,s)$. The quotient ring $R(\sigma,s)/I_N$  is denoted by $R^c(\sigma,s)$.
\item 
The sequence
\begin{equation}\label{seq.fin}
\xymatrix{0 \ar[r] & R \ar[r]^-{\varphi_R^c} & R^c(\sigma,s) \ar[r]^-{\varphi^c_{\ZZ_N}} & \ZZ_N \ar[r] & 0,}
\end{equation}
where $\varphi^c_R: a\lto (a,0)+ I_N$ and $\varphi^c_{\ZZ_N}: (a,k) +I_N \lto k\!\!\mod\! N$ is an exact sequence of rings.
\end{blist}
\item The set
$
\{(a,1)\;|\; a\in R\} \subseteq R(\sigma,s)
$
(resp.\ $\{(a,1)+I_N\;,|\; a\in R\} \subseteq R^c(\sigma,s)$ in the finite exponent case)
is a sub-truss and a paragon of $\tT(R(\sigma,s))$  (resp.\ $\tT(R^c(\sigma,s))$ in the finite exponent case).
\end{zlist}
\end{theorem}
\begin{proof}
That multiplication \eqref{prod.Z} makes $R\times \ZZ$ into an associative ring can be checked by direct calculations that use the double homothetism rules in Definition~\ref{def.homothetism} and equations \eqref{h.ext.cond}. Explicitly, for all $a,b,c\in R$ and $k,l,m\in \ZZ$,
$$
\begin{aligned}
((a,k)(b,l))(c,m) &= \left(ab +l a\sigma + k\sigma b + kls, kl\right)(c,m)\\
&= (abc + l a\sigma  c + k\sigma bc + klsc + mab \sigma +lm a\sigma^2 \\
& + km\sigma b\sigma  + klm s\sigma +kl\sigma c + klms, klm)\\
&= (abc + l a\sigma  c + k\sigma bc + klsc + mab \sigma +lm a\sigma + lmas \\
& + km\sigma b\sigma + klm \sigma s +kl\sigma c + klms, klm)\\
&= (a,k) \left(bc +m b\sigma + l\sigma c + lms, lm\right) = (a,k)((b,l)(c,m)). 
\end{aligned}
$$
Hence the multiplication is associative. The distributive laws follow from the additivity of $\sigma$ and distributive laws in $R$ and $\ZZ$.
The statement (1b) follows immediately from the definition of $R(\sigma,s)$. 

To prove (1c) first note that the sequence \eqref{seq.z} splits as a sequence of abelian groups, and hence there is  the following diagram with exact rows,
\begin{equation}\label{seq.z.split}
\xymatrix{0 \ar[r] & R \ar@<+0.5ex>[r]^-{\psi_R} & S \ar@<+0.5ex>[l]^-{\zeta}\ar@<+0.5ex>[r]^-{\psi_\ZZ}\ar@(rd,dl)^{\bar\pi} \ar@(lu,ur)^\pi & \ZZ \ar@<+0.5ex>[l]^-{\kappa}\ar[r] & 0,}
\end{equation}
in which $\kappa, \zeta, \pi$ and $\bar\pi$ are additive maps such that
$$
\psi_\ZZ\circ\kappa = \id_\ZZ, \quad \pi= \kappa\circ \psi_\ZZ, \quad \bar\pi = \id_S-\pi = \psi_R\circ\zeta, \quad \zeta\circ\psi_R = \id_R.
$$
For any $q\in  {\varphi_\ZZ}^{-1}(1)$, define a double operator $\sigma$ and $s\in R$, by
\begin{equation}\label{datum}
\sigo a = \zeta(q\psi_R(a)), \qquad a\sigo= \zeta(\psi_R(a)q), \qquad s=\zeta(q^2-q),
\end{equation}
for all $a\in R$. 
Observe that,  for all $a\in R$,
$$
\pi(q^2 - q) =\pi(q\psi_R(a))  = \pi(\psi_R(a)q) =0,
$$
by the exactness of the sequence \eqref{seq.z}. 
Hence equations \eqref{datum} can be equivalently written as
\begin{equation}\label{datum'}
\psi_R(\sigo a)= q\psi_R(a), \qquad \psi_R(a\sigo ) = \psi_R(a)q, \qquad \psi_R(s)=q^2-q.
\end{equation}
Using \eqref{datum'} and the fact that $\psi_R$ is an injective ring homomorphism, one easily finds that conditions \eqref{h.ext.cond} are satisfied for $\sigma$ and $s$ defined by \eqref{datum}, and so there is an exact sequence such as \eqref{seq.inf}.

Again a simple calculation aided by \eqref{datum'}  and the ring monomorphism property of $\varphi_R$ yields that 
$$
\Theta: R(\sigma,s)\lra S, \qquad (a,n)\lto \psi_R(a)+nq, 
$$ 
 is an isomorphism of rings with the inverse $\Theta^{-1}(x) = (\zeta(x), \psi_\ZZ(x))$, which provides one with the required equivalence of extensions.

If $(R,+)$ has a finite exponent $N$, then, for all $a\in R$ and $k,l\in \ZZ$,
$$
(a,k)(0,Nl) = \left(Nl a\sigma + klNs, klN\right) = (0,klN) \in I_N,
$$
and similarly for the right multiplication. Hence $I_N$ is an ideal in $R(\sigma, s)$. 

In view of the definition of $I_N$, the map $\varphi^c_R$ is injective. The kernel of $\varphi^c_{\ZZ_N}$ consists of all elements of the form $(a,0) +I_N$, i.e.\ of the whole of the image of $\varphi^c_R$. The map $\varphi^c_{\ZZ_N}$ is obviously surjective. This proves the exactness of the sequence \eqref{seq.fin}.

Finally, since $(0,1)^2 = (0,1) + (s,0)$ in both cases (modulo $I_N$ in the finite exponent case), these extensions satisfy assumptions of Proposition~\ref{prop.truss.ext} with $q=(0,1)$ (or $q=(0,1)+I_N$ in the cyclic case). The stated subsets are of the form $q + \varphi_R(R)$ and hence they are  trusses and paragons, as claimed. This  completes the proof of the theorem.
\end{proof}

\begin{definition}\label{def.h.ext}
Let $R$ be a ring.
\begin{zlist}
\item A pair $(\sigma,s)$, where $\sigma \in \Pi(R)$ and $s\in R$ satisfying conditions \eqref{h.ext.cond} is called a {\em homothetic datum} on $R$.
\item Let $(\sigma,s)$ be a homothetic datum on $R$. 
\begin{rlist}
\item The extension $(\varphi_R,R(\sigma,s),\varphi_\ZZ)$ given by the sequence \eqref{seq.inf} is  called an {\em integral} or {\em infinite homothetic extension}.
\item  The extension $(\varphi^c_R,R^c(\sigma,s),\varphi^c_{\ZZ_N})$ given by the sequence \eqref{seq.fin} is  called a {\em cyclic} or {\em finite homothetic extension}.
\end{rlist}
\item The image under the canonical (projection) isomorphism $\hH(R)\times\{1\} \lra \hH(R)$ of the truss considered in assertion (3) of Theorem~\ref{thm.h.ext} is called  a {\em homothetic truss} on $R$ and is denoted by  $\tT(\sigma,s)$.
\end{zlist}
\end{definition}

Note that $\tT(\sigma,s)$ is isomorphic to $R$ as a heap and it is a truss since $1$ is an idempotent in $\ZZ$ (or $\ZZ_N$). Explicitly, the multiplication $\di$ in $\tT(\sigma,s)$  is given by
$$
a\di b = ab+a\sigma+\sigma b +s,
$$
for all $a,b\in R$. In particular, for the trivial homothetic datum $(0,0)$ on $R$, $\tT(0,0) = \tT(R)$, the truss associated to the ring $R$. Next, we  identify isomorphism classes of homothetic trusses $\tT(\sigma,s)$.

\begin{lemma}\label{lem.iso.tr}
For all ring automorphisms $\Phi$ of $R$ and all $v\in R$, $\tT(\sigma,s)\cong \tT(\sigma',s')$, where
\begin{equation}\label{trans}
s' = \Phi(s+ v +v^2 - v\sigma - \sigma v) \quad \&\quad  \sigma'=\Phi^*(\sigma - \bar{v}),
\end{equation}
where $\Phi^*$ is the induced bijection on $\Pi(R)$ defined in Lemma~\ref{lem.h.aut}.
\end{lemma}
\begin{proof}
First we need to show that the pair $(\sigma',s')$ is a homothetic datum on $R$. Using the fact that  $(\sigma,s)$ is a homothetic datum, we can compute, 
$$
\begin{aligned}
(s+ v +v^2 &- v\sigma - \sigma v)(\sigma - \bar{v})\\
&= s\sigma  +v^2\sigma+v\sigma -v\sigma^2 -\sigma v\sigma -sv - v^2 -v^3 -v\sigma v +\sigma v^2\\
&=\sigma s  +v^2\sigma  -vs  -\sigma v\sigma -sv - v^2 -v^3 -v\sigma v +\sigma v^2\\
&= \sigma s  +v^2\sigma   -\sigma v\sigma -vs -\sigma^2 v + \sigma v- v^2 -v^3 -v\sigma v +\sigma v^2\\
&= (\sigma - \bar{v})(s+ v +v^2 - v\sigma - \sigma v).
\end{aligned}
$$
Therefore,
$$
\begin{aligned}
s'\sigma' &=  \Phi(s+ v +v^2 - v\sigma - \sigma v)\Phi^*(\sigma - \bar{v})\\
&= \Phi\left((s+ v +v^2 - v\sigma - \sigma v)(\sigma - \bar{v})\right)\\
&= \Phi\left((\sigma - \bar{v})(s+ v +v^2 - v\sigma - \sigma v)\right)\\
&=\Phi^*(\sigma - \bar{v})\Phi(s+ v +v^2 - v\sigma - \sigma v) = \sigma's'.
\end{aligned}
$$
Hence the condition \eqref{h.norm} holds for $s'$ and $\sigma'$. Next, using the fact that $\Phi^*$ preserves both addition and multiplication, property \eqref{adj.inner} and that $\sigma$ and $s$ satisfy \eqref{h.sig} we compute,
$$
\begin{aligned}
{\sigma'}^2 &= \Phi^*(\sigma - \bar{v})^2 =  \Phi^*\left((\sigma - \bar{v})(\sigma - \bar{v})\right) 
\\
&=  \Phi^*\left(\sigma +\bar s - \bar v\sigma - \sigma \bar v + \bar{v}^2\right)\\
&= \Phi^*\left(\sigma -\bar v\right) + \Phi^*\left(\overline{s - v\sigma - \sigma v +v+ {v}^2}\right)\\
&= \sigma' + \overline{\Phi(s - v\sigma - \sigma v +v+ {v}^2)} = \sigma' + \overline{s'},
\end{aligned}
$$
as required. Therefore, there is a homothetic extension $R(\sigma',s')$ and the corresponding truss $\tT(\sigma',s')$.

Consider the map
$$
\Phi_v: R\lra R, \qquad a\lto  \Phi(a+v).
$$
Since $\Phi$ is a ring automorphism  $\Phi_v$ is an automorphism of the heap $\hH(R)$. Furthermore, for all $a,b\in R$,
$$
\begin{aligned}
\Phi_v(a)\Phi_v(b) +\Phi_v(a)\sigma' +\sigma' \Phi_v(b)&+s'
= \Phi(a+v)\Phi(b+v) + \Phi\left((a+v)(\sigma -\bar v)\right)\\
&+ \Phi\left((\sigma-\bar v)(b +v)\right) + \Phi(s - v\sigma - \sigma v+v+ {v}^2)\\
&=  \Phi(ab +a\sigma +\sigma b +s +v) = \Phi_v(ab +a\sigma +\sigma b +s).
\end{aligned}
$$
Therefore, $\Phi_v$ transforms multiplication in $\tT(\sigma,s)$ into multiplication in $\tT(\sigma',s')$, and hence is the required truss isomorphism $\tT(\sigma,s)\cong \tT(\sigma',s')$.
\end{proof}

Note in passing that the second of conditions \eqref{trans} is equivalent to the statement that the outer bitranslations $\xi(\sigo), \xi(\sigo')\in \Xi(R)$ are equivalent in the sense of Definition~\ref{def.out.equiv}.

\begin{lemma}\label{lem.h.aut.rev}
Let $R(\sigma,s)$  and $R(\sigma',s')$ be homothetic extensions of $R$ such that $\tT(\sigma,s)\cong \tT(\sigma',s')$. Then there exists a ring automorphism $\Phi$ of $R$ and an element $v\in R$ such that the relations \eqref{trans} hold.
\end{lemma}
\begin{proof}
Let $\Psi:\tT(\sigma,s) \lra \tT(\sigma',s')$ be a truss isomorphism. 
Define 
$$
\Phi: R\lra R, \quad a\lto \Psi(a) -\Psi(0).
$$
Since $\Psi$ is a heap homomorphism, $\Phi$ is an additive map, as, for all $a,b\in R$,
$$
\Phi(a+b) = \Psi(a - 0 +b) - \Psi(0) = \Psi(a) -\Psi(0) + \Psi(b) -\Psi(0) = \Phi(a)+\Phi(b).
$$
Clearly, $\Phi$ is an automorphism of abelian groups with the inverse $\Phi^{-1}(a) = \Psi^{-1}(a) - \Psi^{-1}(0)$. Set
$$
v = -\Psi^{-1}(0) \in R.
$$
Note that, since $\Psi$ respects the ternary operation, for all $a\in R$,
\begin{equation}\label{psi.a.v}
\Phi(a +v) = \Psi(a - \Psi^{-1}(0) +0) - \Psi(0) = \Psi(a) -\Psi(\Psi^{-1}(0)) +\Psi(0) -\Psi(0) = \Psi(a).
\end{equation}
Since $\tPsi$ is an isomorphism of trusses, for all $a,b\in R$,
\begin{equation}\label{psi.hom}
\Psi^{-1}(a)\Psi^{-1}(b) +\Psi^{-1}(a)\sigma + \sigma\Psi^{-1}(b) +s = \Psi^{-1}(ab +a\sigma' +\sigma' b +s').
\end{equation}
Evaluating the equality \eqref{psi.hom} at $a=b=0$ we obtain,
$$
\Psi^{-1}(s') = s + \Psi^{-1}(0)\sigma + \sigma \Psi^{-1}(0)  + \Psi^{-1}(0)\Psi^{-1}(0)
$$
That is,
\begin{equation}\label{s.psi}
s' = \Psi(s - v\sigma - \sigma v + v^2) = \Phi(s - v\sigma - \sigma v + v+ v^2),
\end{equation}
where the last equality follows by \eqref{psi.a.v}. Therefore, the first of conditions \eqref{trans} holds. Next, setting $a=0$ in \eqref{psi.hom} and using \eqref{s.psi} as well as the fact that $\Psi$ is a homomorphism of heaps, equation \eqref{psi.a.v} and the definitions of $\Phi$,  $\Phi^{-1}$ and $v$ we find
$$
\begin{aligned}
\sigma' b &= \Psi\left(\sigma \Psi^{-1}(b)+\sigma v - v^2 -v - v\Psi^{-1}(b)\right)\\
&=\Phi\left(\sigma \Psi^{-1}(b)+\sigma v - v^2  - v\Psi^{-1}(b)\right)\\
&= \Phi\left(\sigma \Phi^{-1}(b) - v\Phi^{-1}(b)\right) = \Phi^*(\sigma - \bar{v})b.
\end{aligned}
$$
In a similar way by setting $b=0$ in \eqref{psi.hom} one finds that $a\sigma' = a\Phi^*(\sigma - \bar{v})$. Put together these prove that 
$$
\sigma' = \Phi^*(\sigma - \bar{v}),
$$
as required. The proof that $\Phi$ respects multiplication in $R$ follows the same lines as the proof that $\Phi_v$ is a truss homomorphism in Lemma~\ref{lem.iso.tr}
\end{proof}

Our next task is to connect the correspondences between ring and truss isomorphisms described in Lemma~\ref{lem.iso.tr}  and Lemma~\ref{lem.h.aut.rev} with equivalences of homothetic extensions.  Before we do this, however, we would like to make the following observation. Any endomorphism of abelian groups has at least one fixed point, so, in particular, any translation by an element that is not identity, i.e.\ the function $a\lto a+e$, for a fixed $e\neq 0$, is not a group endomorphism. Consequently, there are no translation ring endomorphisms other than the identity. In contrast, given a heap $H$, for any two elements $e,f\in H$, the {\em translation map}
\begin{equation}\label{swap}
\tau_{e}^{f}:H\lra H,\qquad a\lto [a,e,f],
\end{equation}
is a heap automorphism. This leads us to the following definition. 

\begin{definition}
Let $T$ and $T'$ be trusses on the same heap. We say that $T$ and $T'$ are {\it translationally isomorphic}  
if there exist elements $e,f$ such that the translation heap automorphism $\tau_{e}^{f}$ 
is an isomorphism of trusses. In that case we write $T\treq T'$.
\end{definition}

Translationally isomorphic trusses turn out   to play a key role in the study of equivalence classes of homothetic ring extensions.

\begin{theorem}\label{thm.h.equiv}
For any ring $R$,
\begin{zlist}
\item Two homothetic extensions of $R$ are weakly equivalent if and only if the corresponding trusses are isomorphic.
\item Two homothetic extensions of $R$ are equivalent if and only if the corresponding trusses differ by a translation.
\end{zlist}
\end{theorem}
\begin{proof}
For statement (1) we observe that if homothetic extensions $(\varphi_R, R(\sigma,s),\varphi_\ZZ)$ and $(\varphi'_R,R(\sigma',s'),\varphi'_\ZZ)$    are weakly equivalent by $\Theta$ and $\Theta_R$, then :
$$
\Theta \circ \varphi_R=\varphi'_R\circ\Theta_R\qquad \text{implies}\qquad \Theta(a,0)=(\Theta_R(a),0),
$$
$$
 \varphi_\ZZ=\varphi'_\ZZ\circ\Theta\qquad \text{implies}\qquad \Theta(0,1)=(e,1),
$$
for all $a\in R$ and some $e\in R$. Hence
$$
\Theta(a,1)=\Theta(a,0)+\Theta(0,1)=(\Theta_R(a)+e,1),
$$
i.e., the ring isomorphism $\Theta$ induces an isomorphism of trusses 
$$
\Theta_T: T(\sigma,s)\lra T(\sigma',s'), \qquad a \lto \Theta_R(a) +e.
$$

In the opposite direction, observe that Lemma~\ref{lem.iso.tr} and Lemma~\ref{lem.h.aut.rev} establish a bijective correspondence between isomorphisms of trusses $\tT(\sigma,s)$ and systems $(\Phi, v)$. Let us assume that $\tT(\sigma,s)\cong \tT(\sigma',s')$ and let $\Phi$ be the corresponding automorphism of $R$ and $v$ the corresponding element of $R$. Define
$$
\Theta: R(\sigma,s)\lra R(\sigma',s'), \qquad (a,k)\lto (\Phi(a) +k \Phi(v),k).
$$
Clearly, $\Theta$ is an isomorphism of abelian groups, so we need to check whether it preserves multiplications. Note that, for all $a,b\in R$ and $k,l\in \ZZ$,
$$
\sigma' (\Phi(b)-l\Phi(v)) = \Phi^*(\sigma - \bar v)(\Phi(b)-l\Phi(v)) = \Phi(\sigma b +l\sigma v - vb - lv^2),
$$
and, similarly,
$$
(\Phi(a)-k\Phi(v))\sigma' = \Phi(a\sigma +kv\sigma -av - kv^2).
$$
Therefore, in view of the definition of $s'$ in \eqref{trans},
$$
\begin{aligned}
\Theta(a,k)\Theta(b,l) &= ((\Phi(a) +k\Phi(v))(\Phi(b) +l\Phi(v))\\
&+ k\sigma' (\Phi(b)-l\Phi(v))
+ l(\Phi(a)-k\Phi(v))\sigma' +kls',kl)\\
&= (\Phi(ab +kvb +lav +klv^2 +k(\sigma b +l\sigma v - vb - lv^2)\\
& + l(a\sigma +kv\sigma -av - kv^2) +kl(s+v+v^2 - v\sigma - \sigma v)), kl)\\
&= (\Phi(ab  +k\sigma b 
 + la\sigma   +kls+klv ), kl) = \Theta\left(\left(a,k\right)\left(b,l\right)\right),
\end{aligned}
$$
as needed. Finally, setting $\Theta_R = \Phi$  we obtain the required weak equivalence of infinite homothetic extensions (in the sense of Definition~\ref{def.ring.ext}).

For the statement (2), if $(\varphi_R, R(\sigma,s),\varphi_\ZZ)$ and $(\varphi'_R,R(\sigma',s'),\varphi'_\ZZ)$  are equivalent by $\Theta$, then  we can follow arguments of the proof of (1) with $\Theta_R=\id_R$, to obtain the translational isomorphism of trusses 
$$
\Theta_T(a) =a+e = [a,0,e] = \tau_0^e(a).
$$
Hence $T(\sigma,s)\treq T(\sigma',s')$.

If  $T(\sigma,s)\treq T(\sigma',s')$, then there exists $e\in R$ such that 
$\tau_0^e$ 
is an isomorphism of trusses and so, by the argument of the proof of Lemma~\ref{lem.h.aut.rev} the corresponding isomorphism of rings is $\Phi(a)=\tau_0^e(a)-\tau_0^e(0)=a$, and thus 
$$
\Theta:R(\sigma,s)\lra R(\sigma',s'), \qquad (a,k)\lto (a+kv,k),
$$
is the corresponding equivalence of ring extensions.

The cyclic homothetic extension case is treated in exactly the same way.
\end{proof}

We note in passing that  statement (2) of Theorem~\ref{thm.h.equiv} can be viewed as translational invariance of equivalence classes of ideal ring extensions by $\ZZ$ as modifications of homothetic trusses by a translational isomorphism does not lead one out of the equivalence class of the corresponding extension.

We end this section by studying $\tT(\sigma,s)$ in some special cases or with additional properties.

\begin{proposition}\label{prop.hom.tr}
Let $(\sigma,s)$ be a homothetic datum on a ring $R$. 
\begin{zlist}
\item The truss $\tT(\sigma,s)$ is commutative if and only if $R$ is commutative and $\sigma$ is a central double homothetism, that is, $\sigol=\sigor$.
\item The truss $\tT(\sigma,s)$ has an absorber if and only if $\sigma$ is an inner double homothetism. In this case, the ring retract of  $\tT(\sigma,s)$ is isomorphic to $R$.
\item The truss $\tT(\sigma,s)$ has an identity if and only if $\tT(\sigo,s) \cong \tT(\id,0)$.
\item $R$ has identity if and only if $\tT(\id,0) \cong \tT(R)$. In this case,  $\tT(\sigo,s)\cong \tT(R)$.
\end{zlist}
\end{proposition}
\begin{proof}
The truss $\tT(\sigma,s)$ is commutative if, and only if, for all $a,b\in R$,
\begin{equation}\label{comm.tr}
ab + a\sigma +\sigma b = ba +b\sigma +\sigma a.
\end{equation}
Setting either $a=0$ or $b=0$ we obtain the centrality of $\sigma$ and then the commutativity of $R$ follows. In the converse direction \eqref{comm.tr} is obviously satisfied. This proves statement (1).

An element $e\in R$ is an absorber in $\tT(\sigma,s)$ if and only if, for all $a\in R$,
\begin{equation}\label{abs.tr}
ae + a\sigma +\sigma e +s =e \quad \& \quad ea + e\sigma +\sigma a +s = e.
\end{equation}
Setting $a=0$ in \eqref{abs.tr} we obtain that $\sigma e +s =e$ and $e\sigma +s =e$, so plugging these back into \eqref{abs.tr} we conclude that $\sigma = \overline{-e}$. Conversely, if $\sigma = \bar{b}$ for some $b\in R$, then the equality \eqref{h.sig} implies that $\overline{b}^2 = \bar{b} +\bar{s}$ and using this relation one easily checks that equations \eqref{abs.tr} are satisfied with $e=-b^2+s$.

If $e$ is an absorber then, the map $R\lra R$, $r\lto r-e$ is the required isomorphism of rings $\rR(\tT(\sigma,s);e)\cong R$. This completes the proof of statement (2).

An element $u$ is the identity for $\tT(\sigma,s)$, if and only if, for all $a\in R$,
\begin{equation}\label{iden.tr}
au+a\sigma +\sigma u + s = a = ua +u\sigma +\sigma a +s.
\end{equation}
Setting $a=0$ we obtain $s=-\sigma u = - u\sigma$, and thus the existence of the identity $u$ implies that 
$$
a\sigma = a-au \quad \&\quad \sigma a = a - ua.
$$
In other words, $\sigma = \id-\bar{u}$. Now setting $a=u$ in \eqref{iden.tr}, we obtain
$$
0=s+u^2 - u +u\sigma +\sigma u,
$$
and hence $\Phi=\id$ and $v=-u$ induce the required isomorphism of trusses $\tT(\sigo,s) \cong \tT(\id,0)$. In the converse direction, if an automorphism $\Phi$ and element $v$ of $R$ are such that 
$$
\sigma = \Phi^*(\id -\bar{v}) = \id - \overline{\Phi(v)}
$$
and
$$
s = \Phi(v^2 +v-v\sigma -\sigma v) = \Phi(v^2-v),
$$
then $u=\Phi(v)$ is the identity in the truss $\tT(\sigma, s)$. Therefore, statement (3) holds.

Finally, if $R$ has the identity $1$, then every homothetism is inner as $\sigma = \overline{\sigma 1} = \overline{1\sigma}$. Then setting $v = {\sigma 1}$ and $\Phi = \id$ we obtain the required isomorphism of homothetic trusses $\tT(\sigma,s) \cong \tT(0,0) = \tT(R)$. In particular, $\tT(\id,0)\cong \tT(R)$. Conversely, if this last isomorphism holds, then $\id = -\overline{\Phi(v)}$, for some $v\in R$ and an automorphism $\Phi$ of $R$, which means precisely that $-\Phi(v)$ is the identity for $R$.
\end{proof}

\begin{proposition}\label{prop.prod}
Let $A$ be a ring with identity and $B$ be any ring. Then any homothetic truss on the product ring $R=A\times B$ is isomorphic to the product truss $\tT(A)\times \tT(\sigma_B,s_B)$, for some homothetic datum on $B$.
\end{proposition}
\begin{proof}
Let $(\sigma, s)$ be a homothetic datum on $R = A\times B$. Then, for all $(a,b)\in R$, 
$$
\sigo (a,b) = (\sigor\su1 (a,b), \sigor\su 2 (a,b)), \qquad (a,b) \sigo  = (\sigol\su1 (a,b), \sigol\su 2 (a,b)),
$$
for some additive functions $\sigor\su1,\sigol\su 1:A\times B\lra A$, and $\sigor\su2,\sigol\su 2:A\times B\lra B$. Since $\sigo$ is a right operator,
$$
\sigo(a,0) = \sigo(a,0)(1,0) = (\sigor\su1 (a,0), \sigor\su 2 (a,0))(1,0) = (\sigor\su1 (a,0), 0).
$$
 Hence $\sigor\su 2(a,0) = 0$. Furthermore,
$$
(0,0) = \sigo (0,0) = \sigo (0,b)(1,0) = (\sigor\su1 (0,b), \sigor\su 2 (0,b))(1,0) = (\sigor\su1 (0,b), 0),
$$
which implies that $\sigor\su1 (0,b) =0$. Therefore,
$$
\sigo (a,b) = \sigo(a,0) +\sigo(0,b) = (\sigor\su1 (a,0), \sigor\su 2 (0,b)).
$$
In a similar way,
$$
(a,b)\sigo = (\sigol\su1 (a,0), \sigol\su 2 (0,b)).
$$
We thus conclude that 
$$
\sigo = (\sigo_A,\sigo_B),
$$
where $\sigo_A$ is a double operator on $A$ and $\sigo_B$ is a double operator on $B$ given by 
$$
\sigo_Aa = \sigor \su1(a,0),\; a\sigo_A = \sigol \su1(a,0), \qquad \sigo_Bb = \sigor \su2(0,b), \; b\sigo_B = \sigol \su2(0,b).
$$
Hence the problem of constructing homothetic trusses on $R$ splits into the problems of such constructions on $A$ and $B$ separately. Since $A$ has the identity, by statement (4) in Proposition~\ref{prop.hom.tr} all homothetic trusses $A$ are isomorphic to $\tT(A)$, and thus the assertion follows. 
\end{proof}

\section{From trusses to ring extensions}\label{sec.truss-ext}
In this section we start with a truss $T$ and first assign a ring to it and then homothetic extension of this ring in which the truss is contained as in Theorem~\ref{thm.h.ext}. This is achieved in two steps. First we associate a ring $\rR(T;e)$ to any truss $T$ (or, more generally a paragon) and any element $e$ in this truss (not necessarily an absorber or a central element as in \cite[Corollary~5.2]{Brz:tru} or \cite[Lemma~3.14]{Brz:par}). Next we show that there is a homothetic datum $(\sigma, s)$ on $\rR(T;e)$ stemming from the internal structure of the truss $T$ such that the induced truss coincides with the original $T$, that is, $\tT(\sigma, s)=T$.

Let $(T, [-,-,-], \cdot )$ be a truss. Given an element $e\in T$, the induced actions of $T$ on itself are defined as follows
\begin{subequations}\label{ind}
\begin{equation}\label{ind.l}
a\la e b = \lambda^e(a,b) :=  [ab,ae,e],
\end{equation}
\begin{equation}\label{ind.r}
a \ra e b = \varrho^e(a,b) := [ab,eb,e],
\end{equation}
\end{subequations}
for all $a,b\in T$. As shown in \cite{Brz:par} both $\lambda^e$ and $\varrho^e$ are heap homomorphisms in both arguments, which means that both $\la e$ and $\ra e$ distribute over the heap operation. Furthermore, $\lambda^e$ is a left action of the semigroup $(T,\cdot)$ while $\varrho^e$ is a right action. In addition the actions commute or satisfy the bimodule associative law, that is, for all $a,b,c\in T$, 
\begin{equation}\label{bimodule}
a\la e (b\ra e c) = (a\la e b)\ra e c.
\end{equation}
Indeed, 
$$
\begin{aligned}
a\la e (b\ra e c) &= a \la e[bc,ec,e] = [a[bc,ec,e], ae,  e] = [abc,aec,e],
\end{aligned}
$$
by the distributive law and \eqref{assoc}-\eqref{Malcev}. On the other hand and by the same token,
$$
(a\la e b)\ra e c = [ab,ae,e]\ra e c = [ [ab,ae,e] c,  ec, e] =  [abc,aec,e],
$$
as required. Therefore, we can write $a\la e b\ra e c$ for both ways of mixing the actions.

Finally, the Mal'cev identities imply that $e$ absorbs  induced actions, that is, for all $a\in T$,
\begin{equation}\label{absorb}
a\la e e = e\ra e a = e.
\end{equation}

\begin{definition}\label{def.closed}
 Let $(T,[-,-,-], \cdot)$ be a truss. A sub-heap $S$ of $(T,[-,-,-])$ is said to be {\em left-closed} (respectively, {\em right-closed}) if, there exists $e\in S$ such that
\begin{equation}\label{closed}
 \lambda^e(S,S)\subseteq S \qquad \mbox{(respectively,  $\varrho^e(S,S)\subseteq S$)}.
\end{equation}
 \end{definition}
 
 \begin{remark}\label{rem.closed}
 Note that the existential quantifier in Definition~\ref{def.closed} can be replaced by the universal one. Indeed, if the condition \eqref{closed} is satisfied for all $e\in S\neq \emptyset$, then such an $e$ exists.  Conversely, if \eqref{closed} is satisfied for some $e\in S$, then for any $e',s,s'\in S$,
 $$
 s\la{e'}s' = [ss',se',e'] = [[ss',se,e],[se',se,e],e'] = [s\la e s',s\la e e',e'] \in S,
 $$
 and similarly for the right action.
  \end{remark}

 Obviously $T$ is both left and right closed. Similarly, if $P$ is a left paragon in $T$, then, by the definition, $\lambda^e(T,P)\subseteq P$, for all $e\in P$, and hence it is left-closed. For the symmetric reason, right paragon is right-closed.

\begin{theorem}\label{thm.truss.ring}
Let $(T,[-,-,-], \cdot)$ be a truss. Let $S$ be a left- or right-closed sub-heap of $T$. For any  $e\in S$, the binary operation
$\pr e$ on $S$ defined by,
\begin{equation}\label{circ}
a\pr e b := [a\la e b, e\la e b, e] = [a\ra e b, a\ra e e, e],
\end{equation}
for all $a,b\in S$, makes the abelian group $(S,+_e)$ into an associative ring. We denote this ring by $\rR(S;e)$. 

For all $e,f\in T$, consider the  translation heap automorphism 
$\tau_e^f: T\lra T$, $a\lto [a,e,f]$; see  \eqref{swap}.
Let $e\in S$ and let $f\in T$ be such that
\begin{equation}\label{f.closed}
\lambda^e(f,S)\subseteq S \qquad \mbox{(respectively,  $\varrho^e(S,f)\subseteq S$)}.
\end{equation}
Then $\tau_e^f(S)$ is a left-closed (respectively, right-closed) sub-heap of $T$ and  $\tau_e^f$ restricts to an isomorphism of rings $\rR(S;e)\lra \rR(\tau_e^f(S);f)$.
\end{theorem}
\begin{proof}
In order to avoid unwieldy expressions that are too hard to read with ease, in what follows we will suppress the indices $e$ in expressions for products, sums and actions, and keep them only in places where an action induced by a different element appears.

First we check the equality of two expressions for $\pr{}$ in equation \eqref{circ}. This follows by the application of the symmetry rule \eqref{symm},
$$
\begin{aligned}
{}[a\la{} b, e\la{} b, e] &= [ab,ae,e,eb, e^2, e,e]
= [ab,eb,e,ae,e^2,e,e] = [a\ra{} b, a\ra{} e, e].
\end{aligned}
$$
As a consequence of this equality the operation $\pr{}$ is a binary operation on $S$ in both cases; if $S$ is left-closed we use the left actions and when $S$ is right-closed we use the right ones.

The distributive law for $\pr{}$ over $+$ follows by the distributive laws of actions, by the absorption rules \eqref{absorb} and the rearrangement rules \eqref{rear}. Explicitly, for all $a,b,c\in S$,
$$
\begin{aligned}
a\pr{}(b+ c) &= a\pr{} [b,e,c]  = [a \la{} [b,e,c], e\la{} [b,e,c], e]\\
&= [a\la{} b, e, a \la{} c, e \la{} b, e, e\la{} c , e]\\
&= [a\la{} b, e \la{} b, a \la{} c, e , e, e\la{} c , e]\\
&= [a\la{} b, e \la{} b, e, e, a \la{} c, e\la{} c , e] = a\pr{} b+ a\pr{} c.
\end{aligned}
$$
The right distributive law follows by symmetry through expressing the multiplication $\pr{}$ in terms of the right induced action. Finally, the associative law for $\pr{}$ is a consequence of the possibility of expressing of this operation in two different ways in \eqref{circ} and the bimodule associative law \eqref{bimodule}. Explicitly, for all $a,b,c\in S$,
$$
\begin{aligned}
a\pr{}(b\pr{} c) &= [a\la{} (b\pr{} c), e \la{} (b\pr{} c), e]\\
&= [a\la{} b \ra{} c, a\la{} b\ra{} e, a\la{} e, e\la{} b\ra{} c, e \la{} b \ra{} e, e\la{} e, e]\\
&= [a\la{} b \ra{} c, a\la{} b\ra{} e, e, e\la{} b\ra{} c, e \la{} b \ra{} e]
\end{aligned}
$$
On the other hand,
$$
\begin{aligned}
(a\pr{} b)\pr{} c &= [(a\pr{} b) \ra{} c, (a\pr{} b) \ra{} e, e]\\
&= [a\la{} b\ra{} c, e\la{} b\ra{} c, e \ra{} c, a\la{} b\ra{} e, e\la{} b\ra{} e, e \ra{} e, e]\\
&= [a\la{} b\ra{} c, e\la{} b\ra{} c, e , a\la{} b\ra{} e, e\la{} b\ra{} e]\\
&= [a\la{} b \ra{} c, a\la{} b\ra{} e, e, e\la{} b\ra{} c, e \la{} b \ra{} e] = a\pr{}(b\pr{} c),
\end{aligned}
$$
as required.

By \cite[Proposition~4.28]{Brz:par},  $\tau_e^f$ is an isomorphism of $T$-modules $(T,\la{})$ and $(T,\la f)$, 
as well as $T$-modules $(T,\ra{})$ and $(T,\ra f)$, 
that is, it is an isomorphism of heaps such that, for all $a, b\in T$,
\begin{equation}\label{tau.mod}
\tau_e^f(a\la{} b) = a\la f \tau_e^f(b) \quad \& \quad \tau_e^f(b\ra{} a) = \tau_e^f(b) \ra f b .
\end{equation}
In particular, it is an isomorphism of groups $\gG(T;e)\lra \gG(T;f)$ and, hence, if $e\in S$, it restricts to the isomorphism of groups $\gG(S;e)\lra \gG(\tau_e^f(S);f)$.  We need to show that $\tau_e^f(S)$ is a closed sub-heap. Assume that $f$ satisfies the first of conditions in \eqref{f.closed}. Then, for all $a,b\in S$,
$$
\tau_e^f(a)\la{} b = [a\la{} b, e\la{} b, f\la{} b] = [a\pr{} b, e,  f\la{} b] \in S.
$$ 
Therefore, by the first of equations \eqref{tau.mod}, 
$$
\tau_e^f(a)\la f \tau_e^f(b) =\tau_e^f (\tau_e^f(a)\la{} b) \in \tau_e^f(S),
$$
and hence $\tau_e^f(S)$ is left-closed. If the other condition in \eqref{f.closed} is satisfied, then we can use the second of the module map properties \eqref{tau.mod} to draw the required conclusion.

To complete the proof we only need to show that $\tau_e^f$ preserves the multiplications. To this end let us take any $a, b\in S$ and compute,
$$
\begin{aligned}
\tau_e^f(a)\pr f\tau_e^f(b) &= [\tau_e^f(a)\la f \tau_e^f(b), f\la f \tau_e^f(b),f]\\
&= [[a,e,f]\la f \tau_e^f(b), f \la f \tau_e^f(b),f]\\
&= [a\la f \tau_e^f(b), e\la f \tau_e^f(b), f\la f \tau_e^f(b), f \la f \tau_e^f(b),f]\\
&= [\tau_e^f(a\la{} b), \tau_e^f (e\la{} b), \tau_e^f(e)]
= \tau_e^f\left([a\la{} b, e\la{} b, e]\right) = \tau_e^f(a\pr{} b), 
\end{aligned}
$$
where we use the cancellation laws \eqref{cancel} and the module map property \eqref{tau.mod} to derive the fourth equality. Therefore, $\tau_e^f$ restricted to $S$ is an isomorphism of rings as asserted.
\end{proof}

Since $T$ is closed, there is a family of isomorphic rings $\rR(T;e)$ labelled by elements $e\in T$. These rings are of the main interest in what follows.

\begin{corollary}\label{cor.ring-type}~\\
\indent (1)  If $e$ is an absorber in a truss $T$, then, for all $a,b\in T$,
$$
a\pr{} b = ab,
$$
in $\rR(T;e)$.

(2) For any ring  $R$, $\rR(\tT(R);0) =R$. Consequently, for all $e\in R$, $\rR(\tT(R);e) \cong R$.

\end{corollary}
\begin{proof}
The first statement follows immediately from the definition of an absorber and the multiplication $\pr{}$ (and the Mal'cev identities), while the second one is a consequence of the first one and the second statement of Theorem~\ref{thm.truss.ring}.
\end{proof}

\begin{definition}\label{def.invariant}
Let $(T,[-,-,-], \cdot)$ be a truss and let $e\in T$. A subgroup $I\leq \gG(T;e)$ is said to be  {\em left invariant} (respectively, {\em right invariant}) if $\lambda^e(e,I)\subseteq I$ (respectively, $\varrho^e(I,e)\subseteq I)$. The set of all left invariant subgroups of $\gG(T;e)$ is denoted by $\linv T e$ (respectively, $\rinv T e$ for right invariant subgroups).
\end{definition}

\begin{lemma}\label{lem.inv}
Let $(T,[-,-,-], \cdot)$ be a truss. 
\begin{zlist}
\item For all natural $n$,  if $I\in \linv T e$, then $\lambda^e(e^n, I)\subseteq I$ (resp.\  if $I\in \rinv T e$, then $\varrho^e(I,e^n)\subseteq I$).  \item Let $I\leq \gG(T;e)$ be a subgroup such that $e^2\in I$. Then $I\in \linv T e$ (respectively, $I\in \rinv T e$) if and only if $eI\subseteq I$ (respectively, $Ie \subseteq I$). 
\end{zlist}
\end{lemma} 
\begin{proof}
The first statement follows by the fact that $\lambda^e$ and $\varrho^e$ are actions of the semigroup $(T,\cdot)$. For the second statement,  $I\in \linv T e$, if and only if, for all $x\in I$ there exist $y\in I$ such that
$$
e\la{} x = [ex,e^2,e] =y,
$$
that is
$$
ex = [ex,e^2,e, e,e^2] = [y, e,e^2] = y+ e^2 \in I,
$$
as required.
\end{proof}

\begin{proposition}\label{prop.ideal}
Let $(T,[-,-,-], \cdot)$ be a truss. Then $P\neq \emptyset$ is a left (respectively, right) paragon in $T$ if and only if, for all $q\in P$, $\tau_q^e(P)$ is a left ideal in $\rR(T;e)$ such that $\tau_q^e(P) \in \linv T e$ (respectively, $\tau_q^e(P)$ is a right ideal in $\rR(T;e)$ such that $\tau_q^e(P) \in \rinv T e$). 
\end{proposition}
\begin{proof}
Assume first that $P$ is a left paragon. Then, since a left paragon in $T$ is the same as an induced submodule of the left regular module $T$, $\tau_q^e(P)$ is a left paragon in $T$ by \cite[Proposition~3.4]{BrzRyb:con}. Furthermore, since $e\in \tau_q^e(P)$ it is a subgroup of $\gG(T;e)$. The paragon property implies that $\tau_q^e(P) \in \linv T e$, and, for the same reason, for all $a\in T$ and $x\in \tau_q^e(P)$,
$$
a\pr{} x = [a\la{} x, e\la{} x, e]\in \tau_q^e(P).
$$
Hence, $\tau_q^e(P)$ is an invariant ideal in $\rR(T;e)$. 

In the converse direction, assume that $P\subseteq T$ is such that, for all $q\in P$, $I:= \tau_q^e(P)$ is an invariant left ideal in $\rR(T;e)$. Then, for all $a\in T$ and $x\in I$,
$$
a\la{} x = [a\la{} x , e\la{} x, e , e, e\la{} x] = [a\pr{} x, e\la{} x,e] \in I,
$$
so $I$ is a paragon in $T$. Since $P= \tau_e^q(I) = \{[x,e,q]\; |\; x\in I\}$, it is a paragon as well by \cite[Proposition~3.4]{BrzRyb:con}.
 \end{proof}
 
Presently, for any truss $T$  we describe a homothetic extension of $\rR(T;e)$ which contains $T$.
  \begin{theorem}\label{thm.hom}
 Let $T$ be a truss and $e\in T$. Define the double operator $\sige$ on the abelian group $\gG(T;e)$ by
\begin{equation}\label{sigma.e}
 \sige a = e\la{} a, \qquad a  \sige = a\ra{} e,
 \end{equation}
 for all $a\in T$. Then
 \begin{zlist}
 \item The pair $(\sige,e^2)$ is a homothetic datum on $\rR(T;e)$.
 \item As trusses, $T=\tT(\sige,e^2)$.
 \end{zlist}
 \end{theorem}
 \begin{proof}
Since $\sige$ is given by truss actions that preserve $e$, both maps 
are 
group endomorphisms of  $\gG(T;e)$. To prove that $\sige$ is a double homothetism the following properties need to be checked, for all $a,b\in T$:
\begin{subequations}\label{homo.all}
\begin{equation}\label{right}
e\la{}(a\pr{} b) = (e\la{} a)\pr{} b,
\end{equation}
\begin{equation}\label{left}
(a\pr{} b)\ra{} e = a\pr{} (b\ra{} e), 
\end{equation}
\begin{equation}\label{mult}
a\pr{} (e\la{} b) = (a\ra{} e)\pr{} b,
\end{equation}
\begin{equation}\label{homo}
e\la{} (a\ra{} e) = (e\la{} a)\ra{} e. 
\end{equation}
\end{subequations}

We start by proving equation \eqref{right}. By the bimodule property \eqref{bimodule}, the definition of multiplication $\pr{}$, and the distributivity of actions
$$
\begin{aligned}
e\la{}(a\pr{} b) &= e\la{} [a\ra{} b, a\ra{}  e, e] = [e\la{} (a\ra{} b), e\la{} (a\ra{}  e), e]\\
&= [(e\la{} a)\ra{} b, (e\la{} a)\ra{}  e, e] = (e\la{} a)\pr{} b,
\end{aligned}
$$
as required. The equality \eqref{left} is proven by the same arguments (but using the other equivalent definition of the product $\pr{}$). The multiplier property \eqref{mult} is also proven by direct calculations. On one hand,
$$
\begin{aligned}
a\pr{} (e\la{} b) &= [a\la{} (e\la{} b), e\la{} (e\la{} b), e ]\\
& = [ae\la{} b, e^2\la{} b, e] 
=[aeb, ae^2, e^3, e^2b, e],
\end{aligned}
$$
where we have used the fact that $\la{}$ is a left action and its definition as well as the rearrangement and cancellation properties \eqref{rear}. On the other hand, using analogous properties of the right action we find
$$
\begin{aligned}
(a\ra{} e)\pr{} b &= [(a\ra{} e)\ra{} b, (a\ra{} e)\ra{} b, e ]\\
& = [e\ra{} eb, a\la{} e^2, e] = [aeb,e^2b, e^3, ae^2, e]\\
&=[aeb, ae^2, e^3, e^2b, e] = a\pr{} (e\la{} b),
\end{aligned}
$$
as required. The final double homothetism condition \eqref{homo} is a special case of the bimodule property \eqref{bimodule}. 

Directly by the definition of the actions,
$$
\sige e^2 = e \la{} e^2 = [e^3,e^2,e] = e^2 \ra{} e = e^2 \sige,
$$
hence the first of conditions \eqref{trans} is satisfied. Furthermore, since $\la{}$ is the action, for all $a\in T$,
$$
\sige^2a - \sige a = e^2\la{} a - e\la{} a = [e^2\la{} a, e\la{} a , e] = e^2\pr{}a.
$$
In a similar way $a\sige^2 - a\sige = [a\ra{} e^2, a \ra{} e , e] = a\pr{}e^2$, hence $\sige^2 =\sige +\overline{e^2}$ as required for the second of conditions \eqref{trans}. This proves statement (1). 

Since the heap structure of a retract of a heap is equal to the original heap structure, $T$ and $\tT(\sige,e^2)$ are mutually equal as heaps. Let us denote by $\circ$ the product in $\tT(\sige,e^2)$. Then, for all $a,b\in T$,
$$
\begin{aligned}
a\circ b&= a\pr{} b +a\sige + \sige b + e^2 \\
&= [ab, ae,e,eb, e^2, e, ae, e^2, e, e, eb,e^2, e, e, e^2]\\
&= [ab, ae,e,eb, e^2, e, ae, e^2, eb]\\
&= [ab, ae,ae,eb,e,e,e^2,e^2,eb] = ab,
\end{aligned}
$$
where the fourth and last equalities follow by a repetitive use of the cancellation rule \eqref{cancel}, while the fifth equality follows by the symmetry rule \eqref{symm} under the cyclic permutation of odd indices $(3,7,5)$. Therefore, the trusses are equal as stated.
 \end{proof}
 
 \begin{definition}\label{def.hom.ext.tr}
 Let $T$ be a truss and $e\in T$, then the ring $\rR(T;e)(\eps,e^2)$  is called an {\em infinite homothetic extension ring of $T$} and is denoted by $T(e)$. Similarly, the ring $\rR(T;e)^c(\eps,e^2)$ (if it exists) is called a {\em finite homothetic extension ring of $T$} and is denoted by $T^c(e)$. 
 \end{definition}
 
 In view of the definition of the multiplication $\pr{}$ in $\rR(T;e)$ and since $a\eps = ae - e^2$ and $\eps b = eb - e^2$ in $\gG(T;e)$,
 the product of ring $T(e)$ built on the abelian group $\gG(T;e)\times \ZZ$ has the explicit form
 \begin{equation}\label{te.prod}
 (a,k)(b,l) = (ab + (l-1) ae+(k-1)eb + (k-1)(l-1)e^2, kl),
 \end{equation}
 for all $a,b\in T$ and $k,l\in \ZZ$. 
 
 The results of Part~\ref{part.ext}  can be summarised as the following two statements:
 \begin{zlist}
 \item 
 There is a one-to-one correspondence between isomorphism classes of trusses and weak equivalence classes of extensions of rings by $\ZZ$. Furthermore, up to translational isomorphism infinite homothetic trusses on a ring $R$ are in one-to-one correspondence with equivalence classes of extensions of $R$ by $\ZZ$.
 \item There is a two-way onto correspondence between trusses with a selected element and rings with homothetic data given by 
\begin{equation}\label{corresp}
 (R,\sigma,s)\lto \tT(\sigma,s), \qquad (T,e)\lto \rR(T,e).
\end{equation}
 In particular, every truss is a homothetic truss, that is, it is of the form $\tT(\sigma,s)$ for some ring $R$ and a homothetic datum on $R$. Conversely, every ring can be understood as a ring associated to a truss $T$ with a chosen element $e$.
 \end{zlist}

 \part{Interpretation}\label{part.int}

 \section{Universality of homothetic extensions of trusses}\label{sec.univ}
 It has been explained in \cite{BrzRyb:mod} that there is a method of embedding of any truss $T$ in a ring by appending $T$ with an absorber (zero). This is based on extending the truss multiplication $T$ to the coproduct (direct sum) of $T$ with the singleton truss $\{0\}$, $T\boxplus \{0\}$. Presently we review this procedure in brief.
 
 The heap $T\boxplus \{0\}$ consists of odd-length words in letters in $T$ and $0$ of the following form: fix an element $e\in T$,
\begin{equation}\label{ele}
 0,\;\; a,\;\; a\, e\, 0, \;\; a\, 0\, e\, 0\, e\ldots0\, e,\;\; 0\, a\, 0\, e\, 0\ldots e\, 0, \qquad a\in T;
\end{equation}
 see \cite[Proposition~3.6]{BrzRyb:mod}. Any word is identified with a word obtained by independent permutations of elements in odd positions or even positions (in concord with  \eqref{symm}). The operation is by concatenation of words followed by the removal of any pairs of identical letters placed in consecutive positions and application of the heap operation to any triples of consecutive elements of $T$. The multiplication in $T\boxplus \{0\}$ is defined by the rules, for all $a,b\in T$,
\begin{equation}\label{prod.T0}
 a\cdot b = ab, \quad a\cdot 0 = 0\cdot a =0,
\end{equation}
 and extended to the whole of $T\boxplus \{0\}$ by the truss distributivity. The rules \eqref{prod.T0} imply that $0$ is the absorber in this truss, the multiplication \eqref{prod.T0} makes the abelian group $\gG(T\boxplus \{0\};0)$ into a ring. We denote this ring by $T_0$.  
 Observe that any homomorphism of rings $\varphi: R\lra R'$ as a function is the same as a homomorphism of corresponding trusses $\tT(\varphi):\tT(R)\lra \tT(R')$, therefore whenever we write a composition of a truss homomorphism $\psi:T\lra \tT(R)$ with $\varphi:R\lra R'$, $\varphi\circ \psi$ we think of $\tT(\varphi)\circ \psi$.
 
 \begin{lemma}\label{lem.ext.0}
 Let $T$ be a truss. An extension $T_0$ has the following universal property. For any ring $R$ and a homomorphism of trusses $\varphi: T\lra \tT(R)$ there exists a unique ring homomorphism $\reallywidehat{\varphi}:T_0\lra R$ rendering commutative the following  diagram
 $$
\xymatrix{T \ar[rr]^-{\iota_T} \ar[dr]_-{\varphi}&& \tT(T_0) \ar@{-->}[dl]^-{\exists !\ \tT(\widehat{\varphi}) }
\\
& \tT(R), &}
$$
 where $\iota_T:T\lra \tT(T_0)$ is given by $t\lto t$. A pair $(T_0,\iota_T)$ is a universal arrow (see \cite[Section III.1, Definition]{Mac:lane} ).
  \end{lemma}
 
\begin{proof}
Let us consider the following commutative diagram of morphisms of trusses:
\begin{equation}\label{uni.t0}
\xymatrix{&& \tT(R)&& \cr T \ar[rr]^{\iota_T}\ar[urr]^{\varphi} & &\tT(T_0) \ar@{-->}[u]_{\widetilde{\varphi}} & & \{0\}\ar[ll]_{\tT(\iota_0)}\ar[ull]_{\tT(j)} ,}
\end{equation}
where $j$ and $\iota_0$ are unique ring homomorphisms from the zero object $\{0\}$ in the category of rings. The existence of the unique truss morphism $\widetilde{\varphi}: \tT(T_0) = T\boxplus\{0\}\lra \tT(R)$ follows by the universal property of the coproduct. Since 
$$
\widetilde{\varphi}\circ \tT(\iota_0)(0)=\widetilde{\varphi}(0)=\tT(j)(0)=0_R,
$$ 
$\widetilde{\varphi}=\tT(\reallywidehat{\varphi})$ for some (unique) ring homomorphism $\reallywidehat{\varphi}:T_0\lra R$.
\end{proof}

The following theorem allows one to identify the universal ring extension $T_0$ with an infinite homothetic ring extension $T(e)$ of Definition~\ref{def.hom.ext.tr}.
 
 \begin{theorem}\label{thm.ext.0}
 Let $T$ be a truss and $e\in T$. For any $a\in T$ and $n\in \ZZ$ define the following elements of $T\boxplus \{0\}$:
 \begin{equation}\label{an}
 a[n] = 
 \begin{cases}
 \underbrace{a\, 0\, e\, 0\, e\ldots 0\, e}_{2n-1}\, , & n>0\\
 \underbrace{a\, e\, 0\,  e\, 0\ldots e\, 0}_{-2n+3}\, , & n\leq 0.
 \end{cases}
 \end{equation}
 Then:
 \begin{zlist}
 \item $T\boxplus \{0\} = \{ a[n]\; |\; a\in T, n\in \ZZ\}$.
 \item The map
 $$
 \chi_e: T_0\lra T(e), \qquad a[n]\lto (a,n),
 $$ 
 is a unique isomorphism of rings such that $\chi_e(a) = (a,1)$.
 \end{zlist}
 \end{theorem}
 \begin{proof}
 (1) Note that $e[0]= e\,e\,0 =0$, hence $e[0]=0$, $a[n]$, $n\geq 0$, describe all the elements of the first four types listed in \eqref{ele} (in particular $a[1]=a$). Finally,
 $$
 \begin{aligned}
 0\, a\, 0\, e\, 0\ldots e\, 0 &= e\, e\, 0\, a\, 0\, e\, 0\ldots e\, 0\\
 &= e\, a\, 0\, e\, 0\, e\, 0\ldots e\, 0\\
 &= [e, a, e]\, e\, 0\, e\, 0\, e\, 0\ldots e\, 0 = [e, a, e][n],
 \end{aligned}
 $$
 for the negative $n$ such that $-2n +1$ is equal to the length of the original word. This completes the proof of statement (1).
 
 (2) In view of the assertion of statement (1) and since there are no repetitions of elements listed as $a[n]$, it is clear that the map $\chi_e$ is a bijection. By the universal property of coproduct it is a unique homomorphism of heaps that fits the diagram
 \begin{equation}\label{uni.chi}
\xymatrix{&& \hH(T(e))&& \cr T \ar[rr]^{\iota_T}\ar[urr]^{\iota_{T(e)}} & &T\boxplus \{0\} \ar@{-->}[u]_{\chi_e} & & \{0\}\ar[ll]_{\iota_0}\ar[ull]_{j} ,}
\end{equation}
where $\iota_T: a\lto a=a[1]$, $\iota_0: 0\lto 0= e[0]$,  $j: 0\lto (e,0)$ and $\iota_{T(e)}: a\lto (a,1)$ are homomorphisms of heaps.  Note that the commutativity of the right hand triangle means that $\chi_e(0) = (e,0)$ so the zero of the ring $T_0$ is transformed to the zero of $T(e)$. Hence by the Lemma \ref{lem.ext.0} the map $\chi_e$ is a (unique) homomorphism of rings. 

The additivity of $\chi_e$ can also be checked directly by considering the following four cases. For positive $k,l\in \ZZ$,
 $$
 \begin{aligned}
 a[k]+b[l] &= \underbrace{a\, 0\, e\, 0\, e\ldots 0\, e}_{2k-1} 0 \underbrace{b\, 0\, e\, 0\, e\ldots 0\, e}_{2l-1}\\
 &= \underbrace{a\, 0\, b\, 0\, e\ldots 0\, e}_{2(k+l)-1} =  \underbrace{a\, e\,e\, 0\, b\, 0\, e\ldots 0\, e}_{2(k+l)+1}\\
&= \underbrace{a\, e\, b\, 0\, e 0\, e\ldots 0\, e}_{2(k+l)+1} =  \underbrace{[a, e, b]\, 0\, e\, 0\, e\ldots 0\, e}_{2(k+l)-1}= [a, e, b][{k+l}],
 \end{aligned}
 $$ 
 where we have used the symmetry to swap $e$ with $b$ and other  rules of operations in $T\boxplus\{0\}$. Therefore,
 $$
 \chi_e(a[k]+b[l] ) = ( [a, e, b], k+l) = (a+_e b, k+l) = (a,k)+(b,l) =  \chi_e(a[k])+\chi_e(b[l] ).
 $$ 
 The case of two non-positive indices is treated in a similar way. Now assume that $k>0$ and $l\leq 0$ are such that $k+l>1$. Then, for all $a,b\in T$,
 $$
  \begin{aligned}
 a[k]+b[l] &= \underbrace{a\, 0\, e\, 0\, e\ldots 0\, e}_{2k-1} 0  \underbrace{b\, e\, 0\,  e\, 0\ldots e\, 0}_{-2l+3}\\
 &= \underbrace{a\, 0\, e\, 0\, e\ldots 0\, e}_{2k-1} 0  \underbrace{0\, e\, 0\,  e\, 0\ldots e\, b}_{-2l+3}\\
  &= \underbrace{a\, 0\, e\, 0\, e\ldots 0\, b}_{2(k+l)-1} = \underbrace{a\, 0\, b\, e\, 0\, e\ldots 0}_{2(k+l)-1}\\
  &= \underbrace{a\, e\, e\,  0\,b\, e\ldots 0\, e}_{2(k+l)+1} = \underbrace{[a, e, b]\,0\, e\ldots 0\, e}_{2(k+l)-1}=  [a, e, b][{k+l}],
 \end{aligned}
 $$
 where the second, fourth and sixth equalities use the freedom of swapping elements in positions with matching parities and the remaining equalities use the cancellation of repeated letters and the application of the heap operation in $T$. Therefore, $\chi_e(a[k]+b[l]) = \chi_e(a[k])+\chi_e(b[l])$ also in this case. The case of $k+l\leq 1$ is treated in a similar way. In conclusion, the map $\chi_e$ is an isomorphism of abelian groups.
 
 To check directly that $\chi_e$ preserve multiplication observe that in the ring $T_0$, for all $a,b\in T$ and $k,l\in \ZZ$,
 \begin{equation}\label{an0}
 a[1]\cdot b[1] = ab[1],\quad  a[k]+b[l] = (a+_e b)[k+l],\quad a[k] = a[1] + (k-1)e[1], 
 \end{equation}
 since $a[1] =a$, $+=+_0$ in $T_0$ and by the additivity of $\chi_e$. Here the concatenation $ab$ means the product in $T$. Using \eqref{an0} and with the understanding that $+$ in-between elements of $T$ means $+_e$ (i.e.\ the operation in the retract $\gG(T;e)$), while $+$ in-between elements of $T_0$  means $+_0$, we can thus compute, 
 $$
 \begin{aligned}
 a[k]\cdot b[l] & = (a[1] + (k-1)e[1])\cdot (b[1] + (l-1)e[1]) \\
 &=ab[1] + (l-1)ae[1]+(k-1)eb[1] + (k-1)(l-1)e^2[1]\\
 &=(ab + (l-1) ae+(k-1)eb + (k-1)(l-1)e^2)[kl]\\
 &=(ab + (l-1) ae+(k-1)eb + (k-1)(l-1)e^2)[kl].
 \end{aligned}
 $$ 
 Hence,
 $$
 \begin{aligned}
 \chi_e(a[k]\cdot b[l]) &= (ab + (l-1) ae +(k-1)eb + (k-1)(l-1)e^2,kl)\\
 & = (a,k)(b,l) = \chi_e(a[k])\chi_e( b[l]),
 \end{aligned}
 $$
 by \eqref{te.prod}.
 Therefore, $\chi_e$ is an isomorphism of rings as stated.
 \end{proof}
 
 The identification of the homothetic extension ring $T(e)$ with the ring $T_0$ allows one to reveal the universality of the former.
 
  \begin{corollary}\label{cor.univ}
 For any truss $T$ and $e\in T$, the infinite homothetic extension ring $T(e)$ has the same universal property as $T_0$. That is, 
 for any ring $R$ and a homomorphism of trusses $\varphi: T\lra \tT(R)$ there exists a unique ring homomorphism $\widehat{\varphi}: T(e)\lra R$ such that  $
\tT(\widehat{\varphi}) \circ \iota_{T(e)} = \varphi.
$
\end{corollary}

\begin{lemma}\label{lem.ep}
Let $T$ be a truss and $e\in T$.  The truss homomorphism $\iota_{T(e)}:T\lra \tT(T(e))$ has the following cancellation property. For all  truss homomorphisms $\varphi,\psi:\tT(T(e))\lra U$ such that $\varphi(e,0) = \psi(e,0)$,  
$$
\varphi\circ \iota_{T(e)}=\psi\circ \iota_{T(e)}\ \  \text{implies}\ \  \varphi=\psi.
$$
In particular, if $U=\tT(R)$ for a ring $R$, then for all ring homomorphisms $f,g: T(e)\lra R$, 
$$
\tT(f)\circ \iota_{T(e)}=\tT(g)\circ \iota_{T(e)}\ \  \text{implies}\ \  f=g.
$$
\end{lemma}
\begin{proof}
Recall that elements of $T(e)$ are of the form $(a,k)\in T\times \ZZ$. We will prove the lemma by induction on $k$ (separately for positive and negative integers).  Let $\varphi,\psi:\tT(T(e))\lra U$ be truss morphisms such that $\varphi(u,0) = \psi(u,0)$, for some $u\in T$, and 
$\varphi\circ \iota_{T(e)}=\psi\circ \iota_{T(e)}$. The second condition means that $\varphi(a,1) = \psi(a,1)$, for all $a\in T$. Assume that  $\varphi(a,k) = \psi(a,k)$ for some positive $k\in \ZZ$ and all $a\in T$. Then,
$$
\begin{aligned}
\varphi(a,k+1) &= \varphi\left([a,u,u],k-0 +1\right)  = \varphi\left([(a,k), (u,0), (u,1)]\right)\\
&= \left[\varphi(a,k), \varphi(u,0), \varphi(u,1)\right] = \left[\psi(a,k), \psi(u,0), \psi(u,1)\right] = \psi(a,k+1), 
\end{aligned}
$$
since both $\varphi$ and $\psi$ are heap homomorphisms. Similarly, for all $k\leq 1$, if $\varphi(a,k) = \psi(a,k)$, then 
$$
\begin{aligned}
\varphi(a,k-1) &= \varphi\left([a,u,u],k-1 +0\right)  = \varphi\left([(a,k), (u,1), (u,0)]\right)\\
&= \left[\varphi(a,k), \varphi(u,1), \varphi(u,0)\right] = \left[\psi(a,k), \psi(u,1), \psi(u,0)\right] = \psi(a,k-1). 
\end{aligned}
$$
Therefore, $\varphi = \psi$. Finally, since $(e,0)$ is the zero of the ring $T(e)$ the first condition is satisfied with $u=e$ and hence the second assertion follows.
\end{proof}

Since, for every $e\in T$, $T(e)$ is isomorphic to $T_0$, the truss homomorphism $\iota_T: T\lra \tT(T_0)$ defined in Lemma~\ref{lem.ext.0} has a similar cancellation property:
\begin{corollary}\label{cor.ep}
For all trusses $T$ and for all ring homomorphisms $f,g: T_0\lra R$, if 
$\tT(f)\circ \iota_{T}=\tT(g)\circ \iota_{T}$, then  $f=g$.
\end{corollary}
\begin{proof}
If $\tT(f)\circ \iota_{T}=\tT(g)\circ \iota_{T}$ then $\tT(f)\circ \iota_{T}\circ \tT(\chi_e^{-1})=\tT(g)\circ \iota_{T}\circ \tT(\chi_e^{-1})$, where $\chi_e: T_0\lra T(e)$ is the ring isomorphism constructed in Theorem~\ref{thm.ext.0}. Since $\iota_{T}\circ \tT(\chi_e^{-1}) = \iota_{T(e)}$ by the diagram \eqref{uni.chi}, the assertion follows from Lemma~\ref{lem.ep}.
\end{proof}

The universal property of the ring $T_0$ described in Lemma~\ref{lem.ext.0} gives rise to a functor $(-)_0:\truss \lra \ring$ between categories of trusses and rings, see \cite[Section IV, Theorem 2(ii)]{Mac:lane}. For the sake of keeping the presentation self-contained, we add a short proof. The functor is given for all trusses $T$ by  $T\lto T_0$,  and for all morphisms $\varphi\in \hom{\truss}{T}{U}$ by $\varphi\lto \varphi_0:=\widehat{\iota_{U}\circ \varphi}$, where $\widehat{~~}$ denotes the ring homomorphism induced from a truss homomorphism via the diagram in  Lemma~\ref{lem.ext.0}. Observe that, by Lemma \ref{lem.ext.0}, for all $\varphi\in \hom{\truss}{T}{U}$ and $\psi\in \hom{\truss}{U}{V}$,
 $$
 \begin{aligned}
\tT(\psi_0\circ \varphi_0) \circ \iota_{T} &= 
\tT(\widehat{\iota_{V}\circ \psi})\circ \tT(\widehat{\iota_{U}\circ \varphi})\circ \iota_{T}=\tT(\widehat{\iota_{V}\circ \psi})\circ \iota_{U}\circ \varphi \\
& = 
\iota_{V}\circ \psi\circ \varphi = \tT(\widehat{\iota_{V}\circ \psi\circ \varphi})\circ \iota_{T} = \tT((\psi\circ \varphi)_0) \circ \iota_{T}.
 \end{aligned}
 $$
Lemma~\ref{lem.ep} implies that 
  $$
\psi_0\circ \varphi_0 =   (\psi\circ \varphi)_0.
  $$
 Thus the composition is preserved by the assignment. One can easily check that identity morphisms are preserved. Hence, $(-)_0: \truss\lra \ring$ is a functor.
 
 The following proposition follows by \cite[Section IV, Theorem 2(ii)]{Mac:lane}, but for the sake of the unaccustomed reader, we sketch a proof.
 
 \begin{proposition}\label{prop.adjoint.truss}
 The functor $(-)_0$ is left adjoint to the functor $\tT: \ring\lra \truss$.
 \end{proposition}
 \begin{proof}
 For all trusses $T$ and rings $R$ let us consider the functions
  $$
 \alpha_{T,R}:\hom{\ring}{T_0}{R}\lra \hom{\truss}{T}{\tT(R)}, \qquad 
 f\lto \tT(f)\circ \iota_{T}.
 $$ 
 We will show that these functions define a natural isomorphism of bifunctors $\alpha:  \hom{\ring}{(-)_0}{-}\lra \hom{\truss}{-}{\tT(-)}$.
 
The functions $\alpha_{T,R}$ are injective  by Corollary~\ref{cor.ep}. The universal property in Lemma~\ref{lem.ext.0} immediately implies that the $\alpha_{T,R}$ are also onto. 
For naturality, take any rings $R$, $S$ and trusses $T$, $U$, and consider homomorphisms $f:T_0\lra R$, $\varphi:U\lra T$ and $g:R\lra S$.
 Then
$$
\alpha_{U,R}(f\circ \reallywidehat{\iota_{T}\circ \varphi})=\tT(f\circ\reallywidehat{\iota_{T}\circ \varphi})\circ\iota_{U}=\tT(f)\circ\iota_{T}\circ \varphi=\alpha_{T,R}(f)\circ \varphi,
$$
by Lemma~\ref{lem.ext.0}.
Similarly,
$$
\alpha_{T,S}(g\circ f)=\tT(g\circ f)\circ \iota_{T}=\tT(g)\circ \alpha_{T,R}(f),
$$
as $\tT(g)=g$ as functions.
Therefore $\alpha$ is a natural isomorphism and  the extension to rings functor $(-)_0$ is the left adjoint to $\tT$.
 \end{proof}
 
 Combining Proposition~\ref{prop.adjoint.truss} with Theorem~\ref{thm.ext.0} we thus obtain
 
 \begin{corollary}\label{cor.adjoint.truss}
 For all rings $R$ and trusses $T$, and for all $e\in T$,
 $$
 \hom{\ring}{T(e)}{R}\cong \hom{\truss}{T}{\tT(R)}.
 $$
 \end{corollary}

 \section{Minimality of homothetic extensions of trusses}\label{sec.min}
In Section~\ref{sec.univ} we have described the universal property of the ring $T(e)$ and shown that the assignment of the ring $T_0$ to a truss $T$ is functorial. It seems quite natural that when one extends a truss $T$ to a ring, one would like to obtain as ``small" a ring as possible. Therefore, in this section we will use the universal property to describe different kinds of ``smallness" of truss extensions into rings. To fix notation, whenever we write $\reallywidehat{\varphi}$ we think of a unique filler of the diagram from Corollary \ref{cor.univ}.
 
\begin{definition}\label{def.locsmal}
Let $T$ be a truss, $R$ a ring and let $\eta_{R}:T\lra \tT(R)$ be an injective homomorphism of trusses. We say that
$R$ is a {\em locally small extension of $T$} if there is no subring $S\subsetneq R$ such that  $\eta_R(T)\subseteq S$.
\end{definition}
\begin{proposition}\label{prop.iso.locsmal}
Let $T$ be a truss and $R$ be an extension of $T$ into a ring with injection $\eta_{R}:T\lra \tT(R)$.  Then 
$R$ is a locally small extension 
if and only if $R= \im(\reallywidehat{\eta_R})\cong T(e)/\ker({\reallywidehat{\eta_{R}}})$, for all $e\in T$.
\end{proposition}
\begin{proof}
Let us assume that $R$ is a locally small extension of $T$ 
with $\eta_R: T\lra \tT(R)$ and take any $e\in T$.  
By Corollary \ref{cor.univ}, there exists a unique ring homomorphism $\reallywidehat{\eta_R}:T(e)\lra R$  such that $\tT(\reallywidehat{\eta_R})\circ\iota_{T(e)} = \eta_R$. Consequently, $S=\im(\reallywidehat{\eta_R})$ is a subring of $R$ such that $\eta_R(T) \subseteq S$, and hence $S=R$, by the local smallness of the extension $R$. The first isomorphism theorem for rings yields the required isomorphism.

In the converse direction, let $R={\im(\reallywidehat{\eta_R})}$ (or, equivalently, $R\cong T(e)/\ker({\reallywidehat{\eta_{R}}})$) and suppose that there is a subring $S$ of $R$ such that $\eta_R(T)\subseteq S$.  Let $j: S\lra R$ be the inclusion ring homomorphism and let $\eta_S: T\lra S$ be given by $T(j)\circ \eta_S = \eta_R$. All these maps together with the corresponding ring homomorphisms $\widehat{\eta_R}$ and $\widehat{\eta_S}$ can be fitted in the commutative diagram:
$$
\xymatrix{T \ar[rr]^-{\iota_{T(e)}}\ar[dr]_{\eta_S} \ar@/_2.0pc/@{->}[ddr]_-{\eta_R} && T(e) \ar[dl]^-{\tT(\widehat{\eta_S})} \ar@/^2.0pc/@{->}[ddl]^-{\tT(\widehat{\eta_R})} \\
& \tT(S) \ar[d]_{\tT(j)} &\\
& \tT(R). &}
$$
Hence $j\circ \reallywidehat{\eta_S} = \reallywidehat{\eta_R}$, which implies that $R=\im(\reallywidehat{\eta_R}) \subseteq S$, that is, $S=R$. Therefore, $R$ is a locally small ring extension of $T$.
\end{proof}

\begin{remark}\label{rem.iso.locsmal}
Proposition~\ref{prop.iso.locsmal} indicates that a locally small extension of a truss $T$ into a ring is not necessarily unique (not even up to isomorphism) and also provides one with a method of constructing such extensions. One needs simply to take any ring $R$ which embeds $T$ as a sub-truss of $\tT(R)$ via an inclusion map, say,  $\eta_R$, construct the corresponding unique ring homomorphism $\widehat{\eta_R}: T(e) \lra R$. The ring $S=\im( \widehat{\eta_R}) \subseteq R$ together with the corestriction of $\eta_R$ to $S$ is the required locally small extension (note that $\eta_R(T)\subseteq S$, since $\widehat{\eta_R}\circ \iota_{T(e)} = \eta_R$). The form of the ring map $\widehat{\eta_R}: T(e) \lra R$ can be easily worked out by inductive arguments. Explicitly,
$$
\widehat{\eta_R}: (a,k) \lto \eta_R(a)+ (k-1)\eta_R(e).
$$
In particular, the map $\reallywidehat{\iota_{T(e)}}$ corresponding to the canonical truss inclusion $\iota_{T(e)} : T \lra T(e)$ is equal to the identity map, and hence $T(e)$ is a locally small extension of $T$. In view of Proposition~\ref{prop.iso.locsmal} all other locally small extensions in $T$ correspond to suitable ideals in $T(e)$. 
\end{remark}

\begin{example}\label{ex.integers}
For any integer $r\geq 2$ or $r=-1$ consider the sub-truss of $\tT(\ZZ)$,
$$
T_r = r(r-1)\ZZ + r = \{r((r-1)k + 1)\; |\; k\in \ZZ\}.
$$
 Note that the multiplication in $T_r$ is well-defined since $r^2$ is congruent to $r$ modulo $r(r-1)$. Needless to say, this truss embeds in $\tT(\ZZ)$, with the embedding $\eta: n\lto n$. The map $\eta$ induces a homomorphism of rings $\widehat{\eta}: T_r(r)\lto \ZZ$, which in view of Remark~\ref{rem.iso.locsmal} reads
$$
\widehat{\eta}(r(r-1)k+r,l) = r((r-1)k+l),
$$
for all $k,l\in \ZZ$. Hence the ring $r\ZZ = \im (\widehat{\eta})$ is a locally small ring extension of $T_r$. Since $\im (\widehat{\eta})\subsetneq \ZZ$, $\ZZ$ is not a locally small extension of $T_r$ for all $r\neq -1$.
\end{example}
As indicated in Remark~\ref{rem.iso.locsmal} locally small extensions of a truss $T$ correspond to certain ideals $I$ in $T(e)$. By Theorem \ref{thm.h.ext}(3), $\iota_{T(e)}(T)$ is a paragon in $T(e)$ and thus
$$
\iI(T):=\tau_{(e,1)}^{(e,0)}(\iota_{T(e)}(T))=\{(a,0)\;|\; a\in T\}
$$
is an ideal in the homothetic extension $T(e)$. To ensure that the composite map
$$
\xymatrix{T\ar@{^{(}->}[r] & \iI(T) \ar@{^{(}->}[r] & T(e) \ar@{->>}[r] & T(e)/I},
$$
is an injective map, we need to require that $I$ intersects trivially with $\iI(T)$. In summary, we can state

\begin{lemma}\label{lem:ideal}
Let $T$ be a truss, $I$ be an ideal in $T(e)$ and $\pi:T(e)\lra T(e)/I$ be a canonical epimorphism. Then $T(e)/I$ is a locally small extension of $T$ into a ring with an injection 
$$
\pi\circ \iota_{T(e)}:T\lra T(e)/I
$$
if and only if 
$I\cap \iI(T)=\{(e,0)\}$.
\end{lemma}

\begin{corollary}\label{cor.cyclic}
Let $T$ be a truss such that the (any) retract $\gG(T;e)$ has a finite exponent. Then the cyclic homothetic extension $T^{\mathrm{c}}(e)$ is a locally small extension of $T$.
\end{corollary}
\begin{proof}
Let $N$ be the exponent of $\gG(T;e)$, and let $I_N = \{(e,Nk)\;|\; k\in \ZZ\}$ be the ideal of $T(e)$ that defines the cyclic homothetic extension $T^{\mathrm{c}}(e) = T(e)/I_N$. Then $(e,Nk) \in \iI(T)$ if and only if $k=0$, and hence $T^{\mathrm{c}}(e)$ is a locally small extension as stated.
\end{proof}

In the hierarchy of locally small extensions of a truss $T$ one can distinguish those that are particularly close to $T$.

\begin{definition}\label{def.small}
A locally small extension $(S,\eta_{S})$ of a truss $T$ is called a {\em small extension} provided $\reallywidehat{\eta_S}(\iI(T))$ is an essential ideal in $S$.
\end{definition}

Taking into account the explicit form of the induced ring map $\reallywidehat{\eta_S}$ described in Remark~\ref{rem.iso.locsmal} one immediately obtains the following characterisation of small extensions.
\begin{lemma}\label{lem.small}
Let $T$ be a truss, $e\in T$  and let  $(S,\eta_{S})$ be a  locally small extension of $T$. Then $(S,\eta_{S})$ is a small extension if and only if, for all ideals $J\lhd S$, there exists $a\in T$ such that $a\neq e$ and
$\eta_S(a) - \eta_S(e)\in J$.
\end{lemma}

\begin{example}\label{ex.int.small}
Let $T_r$ be the truss defined in Example~\ref{ex.integers}. Consider the locally small extension $\eta: T_r \lra \tT(r\ZZ)$, $n\mapsto n$ (with $e=r$).  Since $r\ZZ$ is a principal ideal domain, all ideals in $r\ZZ$ are of the form $I_q=qr\ZZ$, for a non-negative integer $q$. Then
$$
I_q\ni qr(r-1) = \eta(r(r-1)q +r) - \eta(r),
$$
and hence $r\ZZ$ is a small extension of $T_r$.
\end{example}

There exist locally small extensions which are not small extensions. As we observed in Remark~\ref{rem.iso.locsmal},  $T(e)$ is a locally small extension of a truss $T$ but usually it is not small. For example if $T=\tT(R)$, for some ring $R$, then $\tT(R)(0) = R\times \ZZ$ and clearly $(R,0)$ is not an essential ideal in $R\times \ZZ$. Similarly, if $(R,+)$ has a finite exponent $N$, then $\tT(R)^{\mathrm{c}}(0) = R\times \ZZ_N$ and e.g.\ $(0,\ZZ_N)$ intersects trivially with $(R,0)$.

Finally we look at extensions which are at the bottom (or top, depending on the point of view) of the hierarchy of locally small extensions.

\begin{definition}\label{def.min}
Let $T$ be a truss and $(S,\eta_S)$ be a locally small extension of $T$ into a ring $S$. Then we say that $(S,\eta_S)$ is a {\em minimal extension} if, for all ideals $I\subseteq T(e)$ such that $I\cap \iI(T)=\{(e,0)\}$,  $I\subseteq \ker(\reallywidehat{\eta_S})$.
\end{definition}
\begin{lemma}
A minimal extension of a truss is unique up to isomorphism.
\end{lemma}
\begin{proof}
 Let $(S,\eta_{S})$ and $(S',\eta_{S'})$ be two minimal extensions of a truss $T$. Then $\ker(\reallywidehat{\eta_S})\subseteq \ker(\reallywidehat{\eta_{S'}})$ and $\ker(\reallywidehat{\eta_{S'}})\subseteq\ker(\reallywidehat{\eta_S})$, so $\ker(\reallywidehat{\eta_S})=\ker(\reallywidehat{\eta_{S'}})$, and isomorphism is given by the first isomorphism theorem for rings.
 \end{proof}
 \begin{example}[Minimal extensions exist] Let $T=\tT(R)$ for a ring $R$, then $R$ is a minimal extension of $\tT(R)$ to a ring. Since $\tT(R)(0) = R\times \ZZ$,  and the ring homomorphism $\widehat{\eta}: R\times \ZZ\lra R$ induced from the identity map $\eta: R\lra R$ is the projection on the first factor. Furthermore, $\iI(T) = (R,0)$, and so  if an ideal $I\lhd R\times \ZZ$ intersects trivially with $(R,0)$ then $I\subseteq (0,\ZZ) = \ker \widehat{\eta}$.
\end{example}

Observe that if $(S,\eta_S)$ is a small extension, then it is a locally small extension. In a similar way every minimal extension is a small extension. This easily follows by the definitions of locally small, small and minimal extensions.

\begin{lemma}\label{lem:minimal}
Let $(S,\eta_S)$ be an extension of a truss $T$ with at least two elements such that $S$ is a domain. If $(S,\eta_S)$ is a small extension, then $(S,\eta_S)$ is a minimal extension.
\end{lemma}
\begin{proof}
Let $I$ be an ideal in $T(e)$ such that $I\cap \iI(T)=\{(e,0)\}$. If $(a,k)\in I$, then for all $b\in T$, 
$
(b,0)(a,k)\in I\cap \iI(T),
$
and hence $(b,0)(a,k)=(e,0)$. Since $\reallywidehat{\eta_S}$ is a ring homomorphism,
$$
\reallywidehat{\eta_S}(b,0)\reallywidehat{\eta_S}(a,k) =\reallywidehat{\eta_S}(e,0)=0,
$$
and the fact that $S$ is a domain implies that
$
\reallywidehat{\eta_S}(b,0)=0$ or $\reallywidehat{\eta_S}(a,k)=0,
$
for all $(a,k)\in I$ and $b\in T$. In particular, for $b\neq e$, 
$\reallywidehat{\eta_S}(b,0) = \eta_S(b) - \eta_S(e) \neq 0$ by the formula in Remark~\ref{rem.iso.locsmal} and since $\eta_S$ is injective. 
Therefore $\reallywidehat{\eta_S}(a,k)=0$, for all $(a,k)\in I$ and thus
$I\subset \mathrm{ker}(\reallywidehat{\eta_S})$ and $(S,\eta_S)$ is a minimal extension.
\end{proof}

\begin{example}
Let us consider trusses $T_r = r(r-1)\ZZ +r$, for $r=-1$ or $r\geq 2$ of Example~\ref{ex.integers}. By Example~\ref{ex.int.small}, $T_r\hookrightarrow \tT(r\ZZ )$ is a small extension and since
$r\ZZ$ is a domain,  $r\ZZ$ is a minimal extension of $T_r$ by Lemma \ref{lem:minimal}.
\end{example}

\begin{example}[Small but not a minimal extension]
Let $p$ be a prime number and consider the truss
$$
T = \begin{pmatrix} 1 & \ZZ_p \cr 0 & 1 \end{pmatrix}
$$
with the usual matrix multiplication and the heap structure arising from the matrix addition. For $e = \begin{pmatrix} 1 &0 \cr 0 & 1 \end{pmatrix}$, the integral extension $T(e)$  can be identified with
$$
T(e)=\left\{\begin{pmatrix} m&a\\0&m\end{pmatrix}\; | \; m\in\mathbb{Z},\, a\in\mathbb{Z}_p\right\},
$$ 
a Dorroh extension of the ring $\ZZ_p$ with zero multiplication. With this identification, $\iota_T: T\lra T(e)$ is the obvious (set-theoretic) inclusion map, and the corresponding ideal $\iI(T)$ of $T(e)$ comes out as
$$
\iI(T)= \begin{pmatrix} 0 & \ZZ_p \cr 0 & 0 \end{pmatrix}.
$$

For all $n\in \NN$, let us define injective truss homomorphisms
$$
\eta_n: T\lra \tT(\ZZ_{p^{n+1}}), \qquad \begin{pmatrix}1&a\\0&1\end{pmatrix} \lto (1-ap^n)\pmod{p^{n+1}}.
$$
The universally constructed ring homomorphisms are
$$
\widehat{\eta}_n: T(e)\lra \ZZ_{p^{n+1}}, \qquad \begin{pmatrix}m&a\\0&m\end{pmatrix} \lto (m-ap^n)\pmod{p^{n+1}}.
$$
Each of the maps $\widehat{\eta_n}$ is onto so the extensions $\eta_n: T\lra \ZZ_{p^{n+1}}$ are locally small. Furthermore, since for all $n$, the ideals
$$
\widehat{\eta}_n(\iI(T)) = \{ap^n\pmod{p^{n+1}}\;|\; a\in \ZZ_p\},
$$
are essential in $\ZZ_{p^{n+1}}$, all these extensions are small. By the uniqueness of the minimal extensions at most one of them could be minimal. Thus we obtain an infinite family of small extensions that are not minimal.
\end{example}

\part{Classifications}\label{part.class}

 \section{Trusses from rings with zero multiplication}\label{sec.zero}
 The results of Section~\ref{sec.ext-truss} allow one to associate a truss to a homothetic datum on a ring, and thus provide one with a way of constructing trusses. In this part we will classify or describe all  trusses induced by homothetic data on rings with particular properties. We start with the simplest possible rings, those with zero multiplication. 
 
 \begin{proposition}\label{prop.hom.data.zero}
 Let $R$ be a ring with  zero multiplication. Then:
 \begin{zlist}
 \item Any homothetic datum on $R$ consists of an element $s\in R$ and a double operator $\sigma$ such that 
 \begin{blist}
 \item for all $a\in R$, $\sigma(a\sigma) = (\sigma a)\sigma$,
 \item $\sigma$ is an idempotent, that is $\sigma^2 =\sigma$,
 \item $s\sigma = \sigma s$.
 \end{blist}
 \item Any truss induced by a homothetic datum is isomorphic to $\tT(\sigma,0)$, where $\sigma$ satisfies conditions (a) and (b) above.
 \item Two trusses $\tT(\sigma,0)$ and $\tT(\sigma',0)$ are isomorphic if an only if there exists an abelian group automorphism $\Phi: R\lra R$ such that $\sigma'=\Phi^*(\sigma)$.
 \end{zlist}
 \end{proposition}
 \begin{proof}
 Since $R$ has zero multiplication, all the bimultiplication conditions \eqref{bimult} are automatically satisfied, so only \eqref{h.comm} remains, and this is precisely condition (1)(a) in the statement of the proposition. For the same reason, $\bar{s}$ is the zero operation, so the homothetic datum conditions reduce to (1)(b) and (1)(c) above.  Any $s\in R$ can be reduced to zero by choosing $\Phi = \id$ and $v= 2\sigo s -s = 2s\sigo -s$ in \eqref{trans}. Indeed, in this case,
 $$
 \begin{aligned}
 s' &= s+v+v^2 -v\sigo -\sigo v = s +  2\sigma s -s -2\sigo s \sigo +s\sigo -2\sigo^2 s +\sigo s =0,
 \end{aligned}
 $$
by properties (1b) and (1c) and since $R$ is a ring with zero multiplication. The assertion (3) follows immediately from Lemma~\ref{lem.iso.tr} and  Lemma~\ref{lem.h.aut.rev}.
 \end{proof}
 
 Our aim in this section is to reveal the contents of Proposition~\ref{prop.hom.data.zero} in a way that could lead to the full classification of trusses built on rings with zero multiplication. Recall first that a double operator $\sigo: R\lra R$ can be identified with an ordered pair of additive endomaps $\sigor$ and $\sigol$ on $R$. Put together, conditions (1a) and (1b) in Proposition~\ref{prop.hom.data.zero} mean that $\sigor$ and $\sigol$ are commuting idempotents in the endomorphism ring $\lend{}{R,+}$. The following lemmas are probably well known. We include them for completeness.

 \begin{lemma}\label{lem.idem}
Let $S$ be a ring and $e,f\in S$ be idempotent elements. The following statements are equivalent:
 \begin{rlist}
 \item $ef= fe$,
 \item there exists exactly one triple of orthogonal idempotents $(e_1,e_2,e_3)$ in $S$ such that
\begin{equation}\label{e.f}
 e= e_1+ e_3\quad \&\quad f=e_2+e_3.
\end{equation}
 \end{rlist}
 \end{lemma}
 \begin{proof}
 If the idempotents $e$ and $f$ commute, then setting 
 $$
 e_1 = e-ef, \quad e_2 = f- fe, \quad e_3 =ef,
 $$
 we obtain a triple of orthogonal indempotent that satisfies \eqref{e.f}. Suppose $(f_1,f_2,f_3)$ is another such triple. Then, since $f_3=ef=fe$, 
 $$
 f_1= e-ef= e_1 \quad \& \quad f_2 = f -fe = e_2,
 $$
 which proves the uniqueness.
 
 In the converse direction  the orthogonality and idempotent property of the $e_i$ imply that $e$ and $f$ are idempotents and that $ef=fe$, as required.
 \end{proof}
 
 \begin{lemma}\label{lem.sub}
 For any abelian group $A$, there is a bijective correspondence between the following sets of data:
  \begin{rlist}
 \item ordered pairs $(\sigor, \sigol)$ of commuting idempotents in the  ring $\lend{}{A,+}$;
 \item ordered triples $(\eps_1,\eps_2,\eps_3)$ of orthogonal idempotents in  $\lend{}{A,+}$;
 \item ordered quadruples $(A_1,A_2,A_3,A_4)$ of subgroups of $A$ such that 
 $$
 A= A_1\oplus A_2\oplus A_3\oplus A_4.
$$ 
  \end{rlist} 
 \end{lemma}
 \begin{proof}
 The equivalence of statements (i) and (ii) is proven in Lemma~\ref{lem.idem}. Given system of orthogonal idempotents in (ii) set $\eps_4 = \id- \eps_1 - \eps_2 -\eps_3$, and then define $A_i =\im \,\eps_i$, $i=1,\ldots, 4$. Conversely, given ordered direct sum decomposition of $A$ as in (iii), set the $\eps_i$, $i=1,2,3$ to be the corresponding projections on the $A_i$. 
  \end{proof}
 
 \begin{remark}\label{rem.idem}
 There is a freedom in setting up ordering of tuples in Lemma~\ref{lem.sub}. In the following examples we choose the convention in which
\begin{blist}
  \item $\sigor$ is a projection on $A_1\oplus A_3$,
 \item $\sigol$ is a projection on $A_2\oplus A_3$.
 \end{blist}
 As a consequence of this choice:
 \begin{blist}
 \item[(c)] $\sigor \circ \sigol$ is a projection on $A_3$,
 \item[(d)] $\sigor - \sigor\circ \sigol$ is a projection on $A_1$,
 \item[(e)] $\sigol - \sigol\circ \sigor$ is a projection on $A_2$,
 \item[(f)] $A_4 = \ker \sigor \cap \ker\sigol$.
 \end{blist}
The identification in (f) follows from the fact that both $\sigor$ and $\sigol$ are sums of orthogonal projections.
 \end{remark}
 
 With all these results at hand we can now describe all trusses corresponding to rings with zero multplication.
 
 \begin{theorem}\label{thm.zero} 
 Let $A$ be an abelian group.
 \begin{zlist}
 \item For any ordered quadruples  $\A := (A_1,A_2,A_3,A_4)$ of abelian subgroups of  $A$ such that $A= A_1\oplus A_2 \oplus A_3 \oplus A_4$, and any elements $s\in A_3 \oplus A_4$, the 
 multiplication 
\begin{equation}\label{prod.idem}
 (a_1+a_2+a_3+a_4)(b_1+b_2+b_3+b_4) = b_1 +a_2 + a_3 +b_3 +s,
\end{equation}
 for all $a_i, b_i\in A_i$, defines a truss on $\hH(A)$.  We denote this truss by $\tT(\A, s)$.
 \item Any truss $\tT(\A, s)$ is isomorphic to $\tT(\A) := \tT(\A, 0)$, and $\tT(\A)\cong \tT(\B)$ if and only if $A_i\cong B_i$, $i=1,\ldots, 4$.  
  \item The product in the homothetic  extension $\tT(\A)(0)$ of the truss $\tT(\A)$ is given by the following formula,
 \begin{equation}\label{prod.idem.exten}
 \left(a_1+a_2+a_3+a_4,k\right)\left(b_1+b_2+b_3+b_4,l\right) = \left(k(b_1 +b_3) +l(a_2 + a_3),kl\right),
\end{equation}
 for all $a_i, b_i\in A_i$ and $k,l\in \ZZ$. 
 \item Any truss arising from or leading to a ring with zero multiplication (in the sense of the correspondence \eqref{corresp})  is of the form  $\tT(\A, s)$.
 \end{zlist}
 \end{theorem}
 \begin{proof}
 If $A$ is equipped with zero multiplication, by Proposition~\ref{prop.hom.data.zero} and  Lemma~\ref{lem.sub}  in conjunction with Remark~\ref{rem.idem}, the pair $(\A, s)$ gives a homothetic datum on this ring. The truss induced by this datum has multiplication \eqref{prod.idem} since
 $$
 \begin{aligned}
 \sigo(b_1+b_2+b_3+b_4) &= \sigor(b_1+b_2+b_3+b_4) = b_1+b_3,\\
 (a_1+a_2+a_3+a_4) \sigo &= \sigol(a_1+a_2+a_3+a_4)= a_2+a_3.
 \end{aligned}
 $$
 Hence $\tT(\A, s)$ is a truss, as claimed. The remaining statements follow immediately from Proposition~\ref{prop.hom.data.zero}, the product formula \eqref{te.prod}, and the correspondence \eqref{corresp}. 
  \end{proof}
 
 \begin{example}\label{ex.indecomp}
 Let $A$ be a nontrivial indecomposable abelian group. Since it cannot be written as a direct sum of two non-trivial subgroups, there are four possible ordered quadruples $\bf{A}$ that necessarily contain one copy of $A$ and three copies of the trivial group $0$. Multiplications in the corresponding trusses  and their homothetic extensions are collected in the following table:
 \\~
 
 \begin{center}
  \begin{tabular}{| c | c |  c | }
\hline
$\bf{A}$  &  Truss $\tT(\mathbf{A})$  & Extension   $\tT({\mathbf{A}})(0)$ \\ 
& $ \forall a,b\in T$ & $ \forall a,b\in T, k,l\in \ZZ$\\
\hline 
$(A,0,0,0)$ &   $ab = b$  &  
  $(a,k)(b,l) = (kb,kl)$ \\\hline
$(0,A,0,0)$ & $ab=a$ & 
$(a,k)(b,l) = (la,kl)$; \\ \hline
 $(0,0, A,0)$ & $ab=a+b$ &  
  $(a,k)(b,l) = (la+kb,kl)$\\ \hline
 $(0,0,0,A)$ & $ab=0$  & 
  $(a,k)(b,l) = (0,kl)$\\ \hline 
 \end{tabular}
 \end{center}
 ~\\
  \end{example}

 \begin{example}\label{ex.vec}
 Let $p$ be a prime number and $n$ be any natural number and set $A$ to be the abelian group 
 $$
 A= \ZZ_p^n= \underbrace{\ZZ_p\oplus \ZZ_p\oplus \ldots \oplus \ZZ_p}_{\mbox{$n$-times}}.
 $$
  Since $\ZZ_p$ is a simple cyclic group all subgroups of $A$ are  isomorphic to $\ZZ_p^k$, $0\leq k \leq n$. Therefore, up to isomorphism, there are as many ordered partitions of $A$ into the direct sum of four subgroups as there are elements  $\mathbf{n} = (n_1,n_2,n_3,n_4) \in \NN^4$ such that $n_1+n_2+n_3+n_4=n$. The corresponding groups are $A_i \cong \ZZ_p^{n_i}$ and they are uniquely determined by  $\mathbf{n}$. Hence, by Theorem~\ref{thm.zero} there are exactly ${n+3}\choose{3}$ non-isomorphic trusses  $\tT(\A)$. In view of the formula \eqref{prod.idem} the product in $\tT(\A)$ comes out as, for all $\mathbf{a}= (a_i)_{i=1}^n,\mathbf{b}= (b_i)_{i=1}^n \in \ZZ_p^n$, 
$$
    \begin{aligned}
\mathbf{a}\mathbf{b} &= \left(a_1+ b_1,\ldots, a_{n_1}+ b_{n_1}, b_{n_1+1}, \ldots  b_{n_1+n_2}, a_{n_1+n_2+1}, \ldots, a_{n_1+n_2+n_3}, 0,\ldots ,0\right)\\
&=\sum_{i=1}^{n_1}a_i \mathbf{e}_i + \sum_{i=1}^{n_1+n_2}b_i\mathbf{e}_i + \sum_{i=n_1+n_2+1}^{n_1+n_2+n_3}a_i\mathbf{e}_i,
\end{aligned}
$$
where the $\mathbf{e}_i$ are members of the standard basis for the $\ZZ_p$-vector space $\ZZ_p^n$.
Therefore, the product in the infinite homothetic extension $\tT(\A)(0)$ is
$$
(\mathbf{a},k)(\mathbf{b},l) = \left(\sum_{i=1}^{n_1}l a_i\mathbf{e}_i + \sum_{i=1}^{n_1+n_2}k b_i\mathbf{e}_i + \sum_{i=n_1+n_2+1}^{n_1+n_2+n_3}l a_i\mathbf{e}_i, kl\right).
$$ 
The ring $\tT(\A)(0)$  can be identified with a particular subring of the ring $M_2(\ZZ_p^n)$  of $2\times 2$ matrices with entries from  the product ring $\ZZ_p^n$ as follows.  Set
$$
\begin{aligned}
\mathbf{u} &= (\underbrace{1,\ldots,1}_{\mbox{$n_1+n_2$}}, \underbrace{0,\ldots,0}_{\mbox{$n_3+n_4$}}), \qquad 
\mathbf{v} = (\underbrace{1,\ldots,1}_{\mbox{$n_1$}},\underbrace{0,\ldots,0}_{\mbox{$n_2$}},\underbrace{1,\ldots,1}_{\mbox{$n_3$}}, \underbrace{0,\ldots,0}_{\mbox{$n_4$}}). 
\end{aligned}
$$
The collection of all upper-triangular matrices 
$$
U:= \left\{
\begin{pmatrix}
 k\mathbf{u}  & \mathbf{a} \cr 0 & k\mathbf{v}
 \end{pmatrix} 
 \mid \mathbf{a}\in \ZZ_p^n, k\in \ZZ\right\}
 $$
 is a subring in $M_2(\ZZ_p^n)$, since both $\mathbf{u}$ and $\mathbf{v}$ are idempotents. The function
 $$
 \tT(\A)(0) \lra U, \qquad (\mathbf{a},k) \lto \begin{pmatrix}
 k\mathbf{u}  & \mathbf{a} \cr 0 & k\mathbf{v}
 \end{pmatrix} ,
 $$
 is the required isomorphism of rings. 
 
 Since the abelian group $\ZZ_p^n$ has the exponent $p$, the cyclic version of the homothetic extension exists. Its description is obtained by replacing $\ZZ$ by $\ZZ_{p}$ in the above formulae.
 \end{example}

 \section{Trusses from rings with trivial annihilators}\label{sec.ann}
 For a ring $R$ we denote by
 $$
\mathfrak{a}(R) = \{a\in R \; |\; \forall r\in R,\; ra = a r =0\},
 $$
 the annihilator ideal of $R$. In this section we describe all trusses that can be associated to a ring with the trivial annihilator. We start by gathering some properties of homothetisms in this case (the first two are well known, see e.g.\ \cite[Section~2]{And:max}).

\begin{lemma}\label{lem.ann.bimul}
Let $R$ be a ring such that $\mathfrak{a}(R)=0$.  Then
\begin{zlist}
\item The map
$$
\beta: R\lra \Omega(R), \qquad a\lto \bar{a},
$$
is injective (a monomorphism of rings).
\item $\Pi(R)=\Omega(R)$.
\item For all $\sigma \in \Omega(R)$ and $s\in R$ such that  $\sigma^2 = \sigma+\bar{s}$,
$$
s\sigma = \sigma s.
$$
\end{zlist}
\end{lemma}
\begin{proof}
If $a,c\in R$ are such that $\bar{a} =\bar{c}$, then, for all $b\in R$,
$$
(a-c)b=0 \quad \& \quad b(a-c) =0,
$$
and thus $a=c$ since $\mathfrak{a}(R)=0$. This proves statement (1). For any $\sigma\in \Omega(R)$ and $a\in R$,  conditions \eqref{bimult} imply that
$$
\overline{(\sigma a)\sigma} = \overline{\sigma (a\sigma)},
$$
and hence \eqref{h.comm} follows by assertion (1), and thus every bimultiplication is self-permutable, i.e.\ a homothetism, as required.

Take $\sigma$ and $s$ such that $\sigma^2-\sigma =\bar{s}$. Then starting with $\bar{s}^2$ and using this relation sufficiently many times one finds that $\bar{s}\sigma = \sigma\bar{s}$. The assertion (3) then follows by \eqref{inn.ideal} and assertion (1).
\end{proof}

\begin{theorem}\label{thm.ann}
Let $R$ be a ring such that $\mathfrak{a}(R)=0$. There is a bijective correspondence between isomorphism classes of homothetic trusses on $R$ and equivalence classes of idempotents in the ring $\Xi(R)$ of outer bimultiplications on $R$ (with respect to the relation defined in Definition~\ref{def.out.equiv}).
\end{theorem}
\begin{proof}
Let $\xi: \Omega(R)\lra \Xi(R)$ be the canonical surjection. 
By Lemma~\ref{lem.ann.bimul}, $(\sigma, s)$ is a homothetic datum on $R$ if and only if $\xi(\sigma)$ is an idempotent in $\Xi(R)$. Let $\sigma,\tau \in \Omega(R) =\Pi(R)$ such that both $\xi(\sigma)$ and $\xi(\tau)$ are  equivalent idempotents. This means that there exist $s,t,v\in R$ and a ring automorphism $\Phi:R\lra R$ such that
\begin{blist}
\item $\sigma^2 - \sigma =\bar{s}$,
\item $\tau^2  -\tau = \bar{t}$,
\item $\tau = \Phi^*(\sigma -\bar{v})$.
\end{blist}
The claim (c) follows from the facts that $\xi(\tau)={\Phi}^\di(\xi(\sigma))$, where $\Phi^\di$ is the ring homomorphism given by the diagram \eqref{diag.phi.di}, and that $\Phi^*$ is an automorphism  of $\Omega(R)$ mapping inner bimultiplications into inner ones. Therefore,
$$
\begin{aligned}
\bar{t} &= \tau^2  -\tau =  \Phi^*((\sigma -\bar{v})^2 - \sigma +\bar{v} )\\
&= \Phi^*(\sigma^2 -\bar{v}\sigma - \sigma\bar{v} +\bar{v}^2 - \sigma +\bar{v} )\\
&= \Phi^*(\bar{s} -\bar{v}\sigma - \sigma\bar{v} +\bar{v}^2 +\bar{v} ) = \overline{\Phi({s} -{v}\sigma - \sigma{v} +{v}^2 +{v} )}.
\end{aligned}
$$
The last equality follows by the combination of \eqref{inn.ideal} and \eqref{adj.inner}. By Lemma~\ref{lem.ann.bimul},
$$
t = \Phi({s} -{v}\sigma - \sigma{v} +{v}^2 +{v} ),
$$
and in view of Lemma~\ref{lem.iso.tr}, $T(\sigma,s)\cong T(\tau,t)$ as required.

In the converse direction, in view of Lemma~\ref{lem.h.aut.rev}, if $T(\sigma,s)\cong T(\tau,t)$ then here exist a ring automorphism $\Phi$ of $R$ and an element $v$ of $R$ such that 
$$
\tau = \Phi^*(\sigma - \bar{v}) = \Phi^*(\sigma) - \Phi^*(\bar{v}) =  \Phi^*(\sigma) - \overline{\Phi(v)},
$$
and hence $\xi(\tau)  ={\Phi}^\di(\xi(\sigma))$, so that the corresponding idempotents in $\Xi(R)$ are equivalent in the sense of Definition~\ref{def.out.equiv}.
\end{proof}

 Theorem~\ref{thm.ann} has a quite surprising consequence for one-sided maximal  ideals in simple rings with identity.
 
 \begin{theorem}\label{thm.simple}
Let $I$ be a maximal right (resp.\ left) ideal in a simple ring with identity $R$. Up to isomorphism there are exactly two trusses  $T$ on the heap $\hH(I)$ such that $I=\rR(T;0)$:
\begin{blist}
\item  $\tT(I)$, i.e.\ the truss with the same multiplication as that in $I$,
 \item  the truss with multiplication given by the formula
$$
I\times I\lra I, \qquad (a,b)\lto ab+a+b.
$$
\end{blist}
\end{theorem}
\begin{proof}
The proof of this theorem relies on a connection between extensions of rings to the ring of bimultiplications and maximal essential ring extensions introduced by Beidar \cite{Bei:ato}, \cite{Bei:ess}. A ring $\mathrm{ME}(I)$ is called a {\em maximal essential extension} of a ring $I$ if $I$ is an essential ideal in  $\mathrm{ME}(I)$, and if $S$ is any ring that contains $I$ as an essential ideal, then there exists a ring homomorphism $\psi: S\lra\mathrm{ME}(I)$  such that $\psi(x)=x$, for all $x\in I$. Beidar proves that if $I$ is any right ideal of a ring $R$ with identity such that $RI=R$, then 
\begin{equation}\label{me}
\mathrm{ME}(I) = \mathrm{Id}(I):= \{a\in R\;|\; aI\subseteq I\},
\end{equation}
the {\em idealiser} of $I$ in $R$, i.e.\ the largest subring of $R$ containing $I$ as an ideal. In particular, the equality \eqref{me} holds for any right ideal in a simple ring $R$ with identity. 

The notion of an idealiser was introduced by Ore \cite{Ore:for} and thoroughly studied by Robson in \cite{Rob:ide}. In particular  in \cite[Proposition~1.1]{Rob:ide} Robson proves an extension of \cite[Satz~1]{Fit:pri}, thus establishing an isomorphism of rings  
$$
\mathrm{Id}(I)/I \cong \End R {R/I}, \qquad a+ I\lto [r+I \lto ar +I].
$$
 
 On the other hand a theorem of Flanigan \cite{Fla:ide} (see \cite{And:max} for an elementary proof) states that a ring $I$ admits a maximal essential extension if and only if it has a trivial annihilator.  In that case $I$ may be identified with the essential ideal $\overline{I}$ of  $\Omega(I)$. If $S$ is any ring that contains ${I}$ as an essential ideal, then the map
 $$
 \psi: S\lra \Omega(I), \qquad s\lto [\bar{s}: (a\lto sa, \; a\lto as)],
 $$
 is a ring homomorphism such that $\psi(a) = \bar{a}$, for $a\in I$. Therefore $\mathrm{ME}(I)\cong \Omega(I)$.
 
 If $I$ is a right ideal in a simple ring $R$, then putting all this information together we obtain the following chain of isomorphisms of rings
 $$
 \Xi(I) = \Omega(I)/\bar{I} \cong \mathrm{ME}(I)/I \cong \mathrm{Id}(I)/I \cong \End R {R/I}.
 $$
 If, furthermore, $I$ is a maximal right ideal, then $R/I$ is a simple right $R$-module, hence, by the Schur Lemma, $\End R {R/I}$ is a division ring, and so is the ring of outer bimultiplications  $\Xi(I)$. Therefore, $\overline{I}$ and $\id +  \overline{I}$ are the only idempotents in  $\Xi(I)$, and by Theorem~\ref{thm.ann}, there are precisely two isomorphism classes of homothetic trusses  on $I$. The corresponding rings (with the zero of $I$ as the neutral element of the additive group) come out as stated in (a) and (b).
 
 The left ideal case is treated in an analogous way.
\end{proof}
 
 We conclude this section with an extensive example which provides one with the full classification (up to isomorphism) of trusses that can be constructed on the heap determined by the abelian group $\ZZ_p\oplus \ZZ_p$, where $p$ is a prime number.
 
 \begin{example}\label{ex.zp.zp}
 It is well-known that there are eight isomorphism classes of rings built on the abelian group $\ZZ_p\oplus \ZZ_p$. These are:
 \begin{rlist}
 \item the field of $p^2$-elements, $\FF_{p^2}$;
 \item the product ring $\ZZ_p\times \ZZ_p$;
 \item the dual numbers ring $\ZZ_p[x]/(x^2)$;
 \item the zero ring $\ZZ_p^0\times \ZZ_p^0$;
 \item the row matrix ring 
 $\begin{pmatrix} \ZZ_p & \ZZ_p \cr 0 & 0\end{pmatrix}$;
 \item the column matrix ring 
 $\begin{pmatrix} 0 & \ZZ_p \cr 0 & \ZZ_p\end{pmatrix}$;
 \item the half-zero ring $\ZZ_p\times \ZZ_p^0$;
 \item the quotient ring $x\ZZ_p[x]/x^3\ZZ_p[x]$.
 \end{rlist}
 \end{example}
 
Rings (i)--(iii) have identity so by Proposition~\ref{prop.hom.tr} each one of them admits exactly one isomorphism class of homothetic trusses. The full classification of homothetic rings on the zero ring in (iv) can be obtained from Theorem~\ref{thm.zero}. Following this theorem we obtain ten such types of rings (the full list is given in the table below). Rings (v) and (vi) are maximal right respectively left ideals in a simple ring with identity, and so each one of them will admit exactly two non-isomorphic homothetic trusses by Theorem~\ref{thm.simple}. The ring (vii) is the product of a ring with identity and the zero ring on an indecomposable group, hence the classification can be obtained by Proposition~\ref{prop.prod} and Example~\ref{ex.indecomp}, and there are four non-isomorphic trusses in this case. Thus it remains only to study the final case (viiii).

The ring $x\ZZ_p[x]/x^3\ZZ_p[x]$ is a commutative $\ZZ_p$-algebra with a basis $x,x^2$ subject to the relation $x^3=0$. Let $\sigma$ be a double homothetism and suppose that 
$$
\sigma x = ax+bx^2 \quad \& \quad x\sigma = cx +dx^2.
$$
Then the module properties imply that $\sigma x^2 = ax^2$ and $x^2\sigma =cx^2$, and 
$(x\sigma)x = x(\sigma x)$ implies that $a=c$. Now, the condition $(\sigma x)\sigma = \sigma(x\sigma)$ is automatically satisfied, and hence any homothetism is of the above form with $a=c$. We thus obtain
$$
\sigma^2 x = a^2x + 2ab x^2 \quad \& \quad x\sigma^2 = a^2x +2ad x^2.
$$
Let $s=s_1x+s_2x^2$. Then the condition $\sigma^2 = \sigma +\bar{s}$ yields,
$$
a^2 =a \quad \& \quad s_1 = 2ab-b = 2ad -d.
$$
Since $\ZZ_p$ is a field,  $b=d$ and there are only the following solutions to these equations: $a=0$, $s_1=-b$  and $a=1$, $s_1=b$. In the first case $\sigma$ is the inner homothetism $\sigma =\overline{bx}$, and so the corresponding truss has the same multiplication as $R$. In the other case, one easily finds that the corresponding truss is isomorphic to $\tT(\id,0)$ by setting $\Phi =\id$ and  $v=bx - s_2x^2$. These two trusses are not isomorphic by Proposition~\ref{prop.hom.tr}. since $x\ZZ_p[x]/x^3\ZZ_p[x]$ has no identity.

Since every truss on a given abelian heap is a homothetic truss on the ring on the corresponding abelian group, we have obtained a full classification of non-isomorphic trusses on the heap $\hH(\ZZ_p\oplus \ZZ_p)$. There are 23 such (classes of) trusses, which we list in the following table.
 
 \begin{center}
  \begin{tabular}{| c | c |  c | }
\hline
& \\
& Product $\di$ in the truss $\tT(\sigma,s)$, \\
Ring $R$ on $\ZZ_p\oplus \ZZ_p$ &    for all $a=(a_1,a_2),\;\; b=(b_1,b_2)\in R$    \\ 
& \\
\hline \hline
&  \\
$\FF_{p^2}$ &   $a\di b = ab$  
 \\ &  \\ \hline
  &  \\
$\ZZ_p\times \ZZ_p$ & $a\di b = ab$ 
\\ &  \\ \hline
&  \\
 $\ZZ_p[x]/(x^2)$ & $a\di b = ab$ \\ 
 &  \\ \hline
  &  $a\di b=0;$ \;\;\;
   $a\di b =b;$ \;\;\;
   $a\di b =a;$ \;\;\;
    $a\di b =a+b;$\\
 $\ZZ_p^0\times \ZZ_p^0$ &
  $a\di b =(a_1,0);$ \;\;\;
  $a\di b =(0,b_1);$ \;\;\;
 $a\di b =(a_1+b_1,b_2);$ \\
 & $a\di b =(a_1+b_1,a_2);$ \;\;\;
  $a\di b =(a_1+b_1,0);$ \;\;\;
  $a\di b =(a_1,b_2)$ \\ \hline 
 &  \\
 $\begin{pmatrix} \ZZ_p & \ZZ_p \cr 0 & 0\end{pmatrix}$ &$a\di b =ab$;  \;\;\; 
  $a\di b =ab+a+b$ \\ 
 &  \\
  \hline 
  &  \\
 $\begin{pmatrix} 0 & \ZZ_p \cr 0 & \ZZ_p\end{pmatrix}$ &$a\di b =ab$;  \;\;\; $a\di b =ab+a+b$ 
  \\ & 
 \\ \hline 
& \\ 
  $\ZZ_p\times \ZZ_p^0$ &$a\di b =(a_1a_2,b_1)$; \;\;\; $a\di b =(a_1a_2,b_2)$; \\ 
  &$a\di b =(a_1a_2,0)$; \;\;\;
  $a\di b =(a_1a_2,b_1+b_2)$\\  & 
\\ \hline 
  &  \\
  $x\ZZ_p[x]/x^3\ZZ_p[x]$ &$a\di b =ab$;  \;\;\;
  $a\di b =ab+a+b$ \\ 
  &  \\ \hline 
 \end{tabular}
 \end{center}
 ~\\

\part*{Coda}\label{part.coda}

\section{Conclusions and outlook}
The main aim of this work was to place a novel theory of trusses on the more familiar ground of classical ring theory. The results presented here show that trusses can be viewed as a different way of dealing with ring extensions, more precisely those that arise from Redei's double homothetisms or ideal extensions of rings by the ring of integers. In short one can make the following heuristic statement: {\bf\em up to translations trusses are equivalence classes of ring extensions by integers}. Despite this closeness of trusses and ideal ring extensions the results of the paper show that trusses cannot be reduced to rings. Our final example, which demonstrates that there are 23 different isomorphism classes of trusses on the abelian group $\ZZ_p\times \ZZ_p$, as opposed to only 8 rings classes on the same group, might serve as a justification for this claim. 

From a universal algebra point of view trusses are not as complicated algebraic systems as rings, let alone, ring extensions, hence, in our opinion should hopefully provide one with quite an effective way of describing such extensions. On the one hand the results of this paper point to applications of trusses to ring theory and homological algebra, while on the other the vast existing knowledge about ring extensions should feed into the theory of trusses.  For example, studying all possible (and, specifically, non-trivial)  ring structures supported on a given abelian group or characterising groups by types of rings which they support are long-standing problems in algebra (see e.g.\ \cite{NajWor:not}). This can now be translated into the classification problem of trusses. Classification of ring extensions, normally undertaken by homological methods,  can be replaced by classification of homothetic trusses on a given ring. Since all extensions of rings can be understood as arising from families of permutable or amicable homothetisms, developing the connections described in this paper to families of homothetisms might produce new techniques for studying more general classes of extensions. Trusses arose as an attempt to understand connections between braces and rings. Once developments presented here are extended to near-rings and their extensions, novel methods of classifying and constructing braces and thus solutions of the set-theoretic Yang-Baxter equation might be obtained. One of the key obstacles to study categories of modules over trusses, such as those encountered in the definition of projectivity \cite{BrzRyb:fun}, arise from the fact that modules over trusses are enriched over the category of abelian heaps, which is not an abelian category. Employing the ring theoretic language of extensions and thus working over an abelian category, might lead to better understanding what modules over a truss really are or how they can be defined in the most effective way.   

In short, we believe that the structural simplicity of trusses and the richness of their connections with other, well-understood, algebraic systems, while allowing one to ``play gracefully with ideas" (Oscar Wilde, {\em De Profundis}), will lead to an enhancement of the algebraic landscape. 

\section*{Acknowledgements}
The research  of the first two authors is partially supported by the National Science Centre, Poland, grant no. 2019/35/B/ST1/01115. We would like to thank Karol Pryszczepko and all other members of the Department of Algebra at the University of Bia\l ystok for interesting discussions.

\end{document}